\documentclass[12pt]{amsart}
\usepackage{amsmath,amsfonts,amssymb,amsxtra,amscd,enumerate,amsthm}
\usepackage{color}
\usepackage{latexsym}
\usepackage{epsfig}
 \usepackage{fullpage}

\newcommand\les{\lesssim}
\newcommand\ges{\gtrsim}

\newcommand\R{\mathbb{R}}

\newcommand\Z{\mathbb{Z}}
\newcommand\N{\mathbb{N}}
\renewcommand\S{\mathbb{S}}

\newcommand{\calQ}{\mathcal Q}
\newcommand{\calT}{\mathcal T}
\newcommand{\calE}{\mathcal E}
\newcommand{\calH}{\mathcal H}
\newcommand{\calC}{\mathcal C}

\newcommand{\calN}{\mathcal N}
\newcommand{\calR}{\mathcal R}
\newcommand{\calL}{\mathcal L}
\newcommand{\bx}{\bold{x}}

\newcommand{\ls}{{\lesssim}}

\newcommand\la{\langle}
\newcommand\ra{\rangle}
\newtheorem{theo}{Theorem}
\numberwithin{theo}{section} 
\newtheorem{lema}[theo]{Lemma}
\newtheorem{prop}[theo]{Proposition}

\newtheorem{defin}[theo]{Definition}

\newtheorem{rem}{Remark}

\newtheorem*{ncon}{Conjecture}

\numberwithin{equation}{section}

\begin{document}
\title{The optimal trilinear restriction estimate for a class of  hypersurfaces with curvature}

\author[I. Bejenaru]{Ioan Bejenaru} \address{Department
  of Mathematics, University of California, San Diego, La Jolla, CA
  92093-0112 USA} \email{ibejenaru@math.ucsd.edu}

\begin{abstract} In \cite{BeCaTa} nearly optimal $L^1$ trilinear restriction estimates in $\R^{n+1}$ 
are established under transversality assumptions only.
In this paper we show that the curvature improves the range of exponents, by establishing $L^p$ estimates,  
for any $p > \frac{2(n+4)}{3(n+2)}$ in the case of double-conic surfaces. The exponent $\frac{2(n+4)}{3(n+2)}$ is shown to be the universal threshold for the trilinear estimate.  

\end{abstract}


\maketitle

\section{Introduction}

For $n \geq 1$, let $U \subset \R^{n}$ be an open, bounded and connected  neighborhood of the origin and let $\Sigma: U \rightarrow \R^{n+1}$ be a smooth parametrization of an $n$-dimensional submanifold of $\R^{n+1}$ (hypersurface), which we denote by $S=\Sigma(U)$. To this we associate the operator $\calE$ defined 
by
\[
\calE f(x) = \int_U e^{i x \cdot \Sigma(\xi)} f(\xi) d\xi. 
\]
Given $k$ smooth, compact hypersurfaces $S_i \subset \R^{n+1}, i=1,..,k$, where $1 \leq k \leq n+1$, the $k$-linear restriction estimate is the following
inequality 
\begin{equation}  \label{MRE}
\| \Pi_{i=1}^k \calE_i f_i \|_{L^p(\R^{n+1})} \les \Pi_{i=1}^k \| f_i \|_{L^2(U_i)}.  
\end{equation}
In a more compact format this estimate is abbreviated as follows:
\[
\calR^*(2 \times ... \times 2 \rightarrow p).
\] 
The fundamental question regarding the above estimate is the value of the optimal $p$ for which it holds true. Given that the 
estimate $\calR^*(2 \times ... \times 2 \rightarrow \infty)$ is trivial, the optimality is translated into the smallest $p$ for which the
estimate holds true. In  \cite{BeCaTa} Bennett, Carbery and Tao clarified the role of transversality between the surfaces involved and established that
under a transversality condition between $S_1,..,S_k$ the optimal exponent is $p=\frac2{k-1}$; the actual result in \cite{BeCaTa} is near-optimal, the optimal problem is currently open. The optimality can be easily revealed by taking $S_i$ to be transversal hyperplanes, in which case the estimate becomes the classical Loomis-Whitney inequality.  

It is also known, in some cases (precisely when $k \leq 2$), or expected, in most of the others, that curvature assumptions improve the range of exponents in \eqref{MRE}, except for the case $k=n+1$. In \cite{BeCaTa}, the authors state that "simple heuristics suggest that the optimal $k$-linear restriction theory requires at least $n+1-k$ non-vanishing principal curvatures, but that further curvature assumptions have no further effect". However, this aspect of the theory is left as an open problem in \cite{BeCaTa}. 
We detail below what is known for each $k$. 

The case $k=1$ has been understood for a very long time. Without any curvature assumptions, the optimal exponent is $p=\infty$; once the surface has some non-vanishing principal curvatures, the exponent improves to $p=\frac{2(l+2)}{l} $, where $l$ is the number of non-vanishing principal curvatures. The case of non-zero Gaussian curvature, corresponding to $l=n$, is the classical result due to Tomas-Stein, see \cite{St}.

The case $k=2$ without any curvature assumptions corresponds to the classical $L^2$ bilinear estimate, where the optimal estimate has been established. 
The case $k=2$ with curvature assumptions brought a twist into the understanding of the role of curvature in this problem. The best possible exponent in 
$\calR^*(2 \times 2 \rightarrow p)$ is $p=\frac{n+3}{n+1}$ and it was conjectured in \cite{FoKl}; it came as a surprise the fact that it was optimal not only for  paraboloids (having $n$ non-zero principal curvatures), but also for cones (having $n-1$ non-zero principal curvatures and a vanishing principal curvature). This problem was intensely studied for most important model surfaces, see \cite{Bou-CM, Wo, Tao-BW, Tao-BP, TV-CM1, Lee-BR,LeeVa}. In \cite{Be2}, we addressed the problem for general surfaces  and revealed a more subtle way in which the curvature comes into play by stating conditions in terms of the shape operators of the surfaces involved.  

The last case for which \eqref{MRE} is fairly well-understood is the case $k=n+1$. We note that in this case, additional curvature assumptions have no effect 
on the optimality of $p$ and this is consistent with the expectation that the optimal $k$-linear restriction theory requires at least $n+1-k$ non-vanishing principal curvatures, which in this case translates into no curvature. It is conjectured that if the hypersurfaces $S_i \subset \R^{n+1}$ are transversal, then \eqref{MRE} holds true for $p \geq p_0=\frac{2}n$. If $S_i$ are transversal  hyperplanes, \eqref{MRE}  is the classical Loomis-Whitney inequality and its proof is elementary.  Once the surfaces are allowed to have non-zero principal curvatures, things become far more complicated and the problem has been the subject of extensive research, see \cite{BeCaTa,Gu-main} and references therein. 
In \cite{BeCaTa}, Bennett, Carbery and Tao establish a near-optimal version of \eqref{MRE}: this is \eqref{MRE} 
with an additional $R^\epsilon$ factor when the estimate is made over balls of radius $R$ in $\R^{n+1}$. The optimal result for \eqref{MRE}, that is without the $\epsilon$-loss, is an open problem; in some cases one can use $\epsilon$-removal techniques to derive the result without the $\epsilon$-loss for $p > \frac2n$, see \cite{BoGu} for the case of surfaces with non-vanishing Gaussian curvature. The end-point for the multilinear Kakeya version of \eqref{MRE} (a slightly weaker statement than \eqref{MRE}) has been established by Guth in \cite{Gu-main} using tools from algebraic topology. 

In the remaining cases, $n \geq 3$ and $3 \leq k \leq n$,  the $k$-linear restriction theory has been addressed in \cite{BeCaTa} where the authors established the near-optimal result for $p \geq \frac{2}{k-1}$. The exponent $\frac2{k-1}$ is sharp for generic surfaces,
but it is not expected to be the optimal exponent once curvature assumptions are brought into the problem. 

In this paper make a first step towards understanding the effect of curvature in the cases $3 \leq k \leq n$. We will argue that the correct conjectured best optimal exponent should be $p(k)=\frac{2(n+1+k)}{k(n+k-1)}$; note that this is strictly smaller than the generic optimal value $\frac{2}{k-1}$. The easy task is to show that one cannot obtain any estimates for $p <p(k)$, regardless of the amount of curvature involved; this is done by using the classical squashed-cap example. The more difficult task is to show that the conjecture is not vacuous by proving it for some class of surfaces: the multi-conical surfaces.

In this paper we look at the trilinear estimate corresponding to $k=3$ and prove the conjecture for a particular class of surfaces: the double-conic ones. These surfaces have the nice property that they have the exact "amount" of curvature to obtain the estimate with the optimal exponent $p(3)$, and no more, in the sense that they are "flat" in the unnecessary directions. This result is comparable to the one of Wolf \cite{Wo} where he proves the bilinear restriction estimate for conical surfaces (which would translate into $1$-conical surfaces in our context). 

In a forthcoming paper we will provide the equivalent result for $4 \leq k \leq  n$ for $k-1$-conical surfaces. The reason to split 
the case $k=3$ from $k\geq 4$ is that the argument for the latter case incurs not only additional technical difficulties, but also conceptual ones.

We now formalize the above introduction and start with the statement of the conjecture. 
\begin{ncon} 
Under appropriate transversality and curvature conditions on the surfaces $S_i$, $\calR^*(2 \times ... \times 2 \rightarrow p)$ 
holds true for any $p \geq p(k)=\frac{2(n+1+k)}{k(n+k-1)}$. 
\end{ncon}

A complete resolution of this problem should have important consequences. The multilinear theory discussed above has had major impact in other problems. We mention a few such examples: In Harmonic Analysis, the bilinear and 
$n+1$ restriction theory was used to improve results in the context of Schr\"odinger maximal function, see \cite{Bou-SMF,Lee-SMF,TV-CM2, DL},  restriction conjecture, see \cite{Tao-BP,BoGu, Gu-res}, the decoupling conjecture, see \cite{BoDe,BoDeGu}. In Partial Differential Equations, the linear theory inspired the Strichartz estimates, see \cite{Tao-book}, while the bilinear restriction theory is used in the context of more sophisticated techniques, such as the profile decomposition, see \cite{MeVe}, and concentration compactness methods, see \cite{KeMe}. We expect that a positive resolution of the Conjecture will have an impact in some of the problems mentioned above. 

We will make an attempt in shedding some light into the appropriate conditions required above, but we believe that only a complete resolution of the Conjecture will clarify the role of geometry into the problem. 

The transversality condition is clear by now: for any choice $\zeta_i \in S_i, i=1,..,k$, if $N_i(\zeta_i)$ is the unit normal to $S_i$, then the 
\[
vol(N_1(\zeta_1),..,N_k(\zeta_k)) \geq \nu >0
\] 
where $vol(N_1(\zeta_1),..,N_k(\zeta_k))$ is the volume of the $k$-dimensional parallelepiped generated by 
the normals $N_i(\zeta_i)$. This is natural to impose, given that it matches the condition required in the generic case where
no curvature is assumed, see \cite{BeCaTa}. 

The more delicate part is to list the curvature conditions on $S_i$. For a hypersurface $S \subset \R^{n+1}$ and $\zeta \in S$, we denote by $N(\zeta)$ the normal to $S$ (assumed to be orientable) at $\zeta$ and by $S_{N(\zeta)}$ the shape operator at $S$ at $\zeta \in S$ with respect to the normal $N(\zeta)$. For an in-depth introduction to the shape operator, we refer the reader to any classical textbook in differential geometry, see for instance \cite{doCa}. For a quick introduction tailored towards the use in the context of bilinear estimates, see \cite{Be2}. Note that we use the same letter $S$ both for the hypersurface and the shape operator: the subscript 
$N$ will indicate that we are referring to the shape operator. 

As for the role of curvature in the above Conjecture, we anticipate a condition of type 
\[
vol(N_1(\zeta_1),..,N_k(\zeta_k), S_{N_j(\zeta_j)} v_{k+1},.., S_{N_j(\zeta_j)} v_{n+1}) \geq \nu >0
\] 
for any choices $\zeta_i \in S_i, i=1,..,k$, for any $j \in \{ 1,..,k \}$, for any choice $\zeta_j \in S_j$ and for any choice 
$v_{k+1},..,v_{n+1}$ is a orthonormal vector basis to specific $n-k+1$-dimensional submanifolds $S \subset S_j$ 
(it suffices to check this for one basis). 
Here $S_{N_i(\zeta_i)}$ is the shape operator at $S_i$ at $\zeta_i \in S_i$ with respect to the normal $N_{i}(\zeta_i)$. 
We expect these submanifolds to be exhausted by the family
\[
S(\beta_1,.., \beta_{j-1},\beta_{j+1},..,\beta_{n+1}) = S_ j \cap_{l \ne j} (\beta_l + S_l).   
\]
for any choices $\beta_1,.., \beta_{j-1},\beta_{j+1},..,\beta_{n+1} \in \R^{n+1}$. But we do not go as far as claiming this to be the ultimate 
curvature condition. 

In Section \ref{pk} we show, by a counterexample, that the range stated in the conjecture is sharp, that is 
\eqref{MRE} fails for $p < \frac{2(n+1+k)}{k(n+k-1)}$ for conic surfaces, but also for quadratic surfaces. This highlights
 that regardless of how much curvature is brought into the problem, the threshold $\frac{2(n+1+k)}{k(n+k-1)}$ is the best exponent
 that can occur in $\calR^*(2 \times ... \times 2 \rightarrow p)$. 

As for positive results, in this paper we look at the case $k=3$ and establish the almost optimal conjectured result for a particular class of hypersurfaces which we introduce below. 

We recall the definition of a foliation. A $2$-dimensional foliation of the ($n$-dimensional) hypersurface $S$ is a decomposition of $S$ into a union of connected disjoint sets $\{ S_\alpha \}_{\alpha \in A}$, called the leaves of the foliation, with the following property: every point in $S$ has a  neighborhood $V$ and local system of coordinates $x: V \subset S \rightarrow \R^{n}$ such that for each leaf $S_{\alpha}$, the coordinates of $V \cap S_{\alpha}$ are $\xi_3=constant,..,\xi_{n}=constant$. 

We are now ready to formalize the conditions we impose on our surfaces. As before, $S_i, \in \{1,2,3 \}$ are hypersurfaces with smooth parameterizations $\Sigma_i: U_i \subset \R^n \rightarrow \R^{n+1}$, where each $U_i$ are open, bounded and connected  neighborhood of the origin (note that different $U_i$ may belong to different hyperplanes identified with the same $\R^n$). In addition, we assume the following three hypothesis:

i) (foliation) for each $i \in \{1,2,3\}$, the hypersurface $S_i$ admits the foliation
\[
S_i = \bigcup_{\alpha} S_{i,\alpha}
\]
where, for each $\alpha$, the leaf $S_{i,\alpha}$ is a flat submanifold of dimension $2$.  
 
ii) (the leaves are completely flat) If $S_{N_i(\zeta_i)}$ is the shape operator of $S_i$ at $\zeta_i \in S_i$ with choice of normal $N_i(\zeta_i)$ we assume that for every $v \in T_{\zeta_i} S_{i,\alpha}$ (the tangent plane at $S_{i,\alpha}$ at the point $\zeta_i \in S_{i,\alpha}$) the following holds
\[
S_{N_i(\zeta_i)} v =0. 
\]

iii) (transversality and curvature) For any $\zeta_i \in S_i, i \in \{1,2,3\}$, for any $l \in \{1,2,3\}$ and for any orthonormal basis 
$v_4,..,v_{k+1} \in (T_{\zeta_l} S_{l,\alpha})^\perp \subset T_{\zeta_l} S_{l,\alpha}$ the following holds true 
\begin{equation} \label{curva}
vol ( N_1(\zeta_1), N_2(\zeta_2), N_3(\zeta_3), S_{N_l(\zeta_l)} v_4 ,  ... , S_{N_l(\zeta_l)} v_{n+1} ) \geq \nu. 
\end{equation}

In \eqref{curva} the $vol$ is the standard volume form of $n+1$ vectors in $\R^{n+1}$, thus the condition quantifies the linear independence of the vectors $N_1(\zeta_1), N_2(\zeta_2), N_3(\zeta_3), S_{N_l(\zeta_l)} v_4 ,  ... , S_{N_l(\zeta_l)} v_{n+1}$. 

The condition ii) says that $S_{i,\alpha}$ are, in some sense, completely flat components of $S_i$ since, besides being subsets of
affine planes of dimension $2$, the normal $N(\zeta)$ to $S_i$ is constant as we vary $\zeta$ along $S_{i,\alpha}$ for fixed $\alpha$. 

The first things to read in condition iii) is the transversality condition between $S_1,S_2$ and $S_3$ due to the transversality between any choice on normals.  The condition iii) also says that the submanifolds transversal to the leafs carry the curvature assumptions, in the sense that their tangent space does not contain any eigenvectors of the shape operator. In addition, for each $i \in \{1,2,3\}$, we are guaranteed to have transversality between $N_1(\zeta_1), N_2(\zeta_2), N_3(\zeta_3)$ and $S_{N_i} T_{\zeta_i} (S_{l,\alpha})^\perp$.

At this point we can state the main result of this paper. 
\begin{theo} \label{mainT}  Assume that $S_1,S_2,S_3$ satisfy the conditions i)-iii) above. 
Given any $p$ with $p(3)=\frac{2(n+4)}{3(n+2)}< p \leq \infty$, the following holds true
\begin{equation} \label{mainE}
\| \Pi_{i=1}^3 \calE_i f_i \|_{L^p(\R^{n+1})} \leq C(p) \Pi_{i=1}^3 \| f_i \|_{L^2(U_i)}, \quad \forall f_i \in L^2(U_i).   
\end{equation}
\end{theo}

To our best knowledge this result is the first instance when the trilinear restriction estimate is shown to improve its range of exponents when curvature assumptions are made; this statement has to be carefully explained, since given large enough $n$ one can obtain trilinear restriction estimates in $L^p$ with $p<1$ based on linear and bilinear restriction estimates, for the latter one see \cite{LeeVa}.  
We are not only going beyond the known $L^1$ estimate, but also provide the sharp result, except for the endpoint $p=p(3)$, and this is truly a feature of the trilinear estimate. In some instances, our result also improves the known results in the range $p \geq 1$, where the standard result is near-optimal, while ours is optimal. Given that the estimate is trivial for $p=\infty$, it suffices to focus on the result above in the cases
$p(3) < p \leq 1$ and this is what we will do.

One aspect that is revealed by the above result is that the optimal trilinear restriction estimate discards the effect of $2$ curvatures; indeed, each $S_i$ has precisely $2$ vanishing principal curvatures. In light of the similar results in the bilinear case, 
this result indicates that the optimal $k$-linear restriction estimate discards the effect of $k-1$
curvatures and relies only on $n+1-k$ curvatures, although the actual statements have to be more rigorous. 

Now we list some hypersurfaces that would fit in our class and for each the above result is applicable. One such model is given by the equation:
\begin{equation} \label{model}
\zeta_{n+1} +  \zeta_{n} = \sqrt{\zeta_1^2+ ...+ \zeta_{n-1}^2}. 
\end{equation}
Away from the origin, this surface has $2$ vanishing principal curvatures and $n-2$ non-vanishing principal curvatures.
If one writes $(\zeta_1,..,\zeta_{n-1})=r \cdot \omega$ with $r=\sqrt{\zeta_1^2+ ...+ \zeta_{n-1}^2}$ and $\omega \in \S^{n-2}$
the above equation becomes $\zeta_{n+1} +  \zeta_{n}=r$, thus revealing the source of the vanishing curvatures as coming from the embedded planes of dimension $2$; these are the leaves of the foliation. 

Another model is given by the cylinders: 
\[
\zeta_1^2+ ...+ \zeta_{n-1}^2=1
\]
where the freedom in choosing $\zeta_{n}$ and $\zeta_{n+1}$ gives rise to the desired structure. 
 
At this time we are not able to establish a similar result for more general surfaces, in particular for hypersurfaces with non-zero Gaussian curvature, such as
the paraboloid $\zeta_{n+1}=\zeta_1^2+ ...+ \zeta_{n}^2$. This may come as a surprise as one expects more curvature to be helpful. By now there is enough body of evidence that additional curvature complicates these problems. Probably the most clear piece of evidence lies in the $n+1$-linear estimate in $R^{n+1}$ (where no curvature is needed for optimal results): 
if the surfaces are transversal hyperplanes, the $L^{\frac1n}$ estimate follows easily from the classical Loomis-Whitney inequality, 
while if the surfaces are assumed to have curvature, the estimate is considerably harder and open.

We now explain some of the key ideas in the proof of Theorem \ref{mainT}. Some of the starting ideas originate in the prior work of the author on the bilinear restriction estimate in \cite{Be2} and the multilinear restriction estimate in \cite{Be1}.

The argument for the bilinear restriction estimate that we provided in \cite{Be2} was inspired by many previous works in the subject, 
and we chose to follow the less common approach of Tao in \cite{Tao-BW}. One of the novelties of our argument is that it 
reveals in a clear way a key feature of the argument: at some stage of it, the proof relies on replicating the argument of the classical bilinear $L^2$ estimate (on the Fourier side, thanks to the Plancherel fromula). To be more specific, inside the body of the argument,
we reproduce a proof of the bilinear estimate
\begin{equation} \label{bil2}
\| \calE_1 f_1 \cdot \calE_2 f_2 \|_{L^2} \les \| f_1 \|_{L^2(U_1)} \| f_2 \|_{L^2(U_2)},
\end{equation}
where transversality assumptions are made on $S_1=\Sigma_1(U_1)$ and $S_2=\Sigma_2(U_2)$. This bilinear estimate is used in a localized setup, meaning  that one quantifies the effect of  $U_1$ or $U_2$ having small support. 
We note that in the combinatorial strategy of proving the bilinear restriction estimate, used for instance in \cite{Tao-BP}, one replaces
the bilinear $L^2$ estimate with its quadralinear version and the relevance of the bilinear $L^2$ estimate \eqref{bil2} is less apparent. 
 
 Thus the natural idea is to follow the same route in the trilinear setup, that is to rely, at some stage of the argument, on the trilinear estimate
\begin{equation} \label{tri2}
\| \calE_1 f_1 \cdot \calE_2 f_2 \cdot \calE_3 f_3 \|_{L^1} \les \| f_1 \|_{L^2(U_1)} \| f_2 \|_{L^2(U_2)} \| f_3 \|_{L^2(U_3)}.
\end{equation}
 This is where one hits an immediate road-block. While the bilinear $L^2$ estimate is well understood and very malleable, with arguments that easily blend into those used in \cite{Be2}, we found the trilinear $L^1$ estimate to be very rigid. In particular, one needs to use the estimate in a  black-box fashion. Moreover, one needs to use a localized version of the trilinear estimate, in the sense that one of the factors has small frequency support, or, equivalently, the surface has small support in some directions. We have not found such a result in literature and probably for a good reason: all the previous proofs for the multilinear restriction estimates seem to go through the multilinear Kakeya and it is not clear how to translate the localized multilinear estimate into a multilinear Kakeya estimate. This was part of the motivation for our work in \cite{Be1} where we provided a direct proof of the multilinear estimate, followed by a refinement in the context of small support. In Section \ref{LocME} we adapt the result from \cite{Be1} to our current needs. 

The next tool is to construct tables adapted to waves which are able to highlight the dispersive effects in outer regions.
This is done by using a wave packet decomposition and as in \cite{Be2}, we use of the refined wave packet decomposition introduced by Tao in \cite{Tao-BW}. This requires a bit more work than the classical wave packet construction, but it has the advantage of providing a direct venue for closing the argument for $p > p(3)$ without additional $\epsilon$-removal techniques, see \cite{TV-CM1} for example. This is important because the $\epsilon$-removal results, see \cite[Lemma 2.4]{TV-CM1}, work for $L^p$ spaces with $p>1$, hence they do not cover the ranges of $L^p$ spaces we are interested in.

We expect to be able to obtain a similar result to the one in Theorem \ref{mainT} for the $k$-linear restriction estimate with $k \geq 4$. The argument provided here fails to work for $k \geq 4$ and we can easily  point to  such an instance: the key estimate \eqref{Phia2} is an $L^1$ estimate and the triangle inequality holds in this space, but it fails in $L^\frac{2}{k-1}$ (and this is the space to be used in the corresponding step for the $k$-linear restriction estimate). But, in fact, there are more subtle points where the argument here simply fails to carry over to the case $k \geq 4$.

\subsection{Notation} \label{not}
We start by clarifying the role of various constants that appear in the argument. 
$N$ is a large integer that depends only on the dimension. $C$ is a large constant that may change
from line to line, may depend on $N$, but not on $c$ and $C_0$ introduced below. $C$ is used in the definition 
of: $A \ls B$, meaning $A \leq C B$,  $A \ll B$, meaning $A \leq C^{-1} B$, and $A \approx B$, meaning $A \les B \wedge B \les A$.
For a given number $r \geq 0$, by $A = O(r)$ we mean that $A \approx r$.  $C_0$ is a constant that is independent of any other constant and its role is to reduce the size of cubes in the inductive argument.  It can be set $C_0=4$ throughout the argument, but we keep it this way so that its role in the argument is not lost. $c \ll 1$ is a very small variable meant to make expressions $\ll 1$ and most estimates will be stated to hold in a range of $c$.

We use the standard notation $(\xi_1,..,\bar \xi_i,..\xi_{l}):= (\xi_1,..,\xi_{i-1}, \xi_{i+1}..,\xi_{l})$ to mean that one component is missing.

By powers of type $R^{\alpha+}$ we mean $R^{\alpha+\epsilon}$ for arbitrary $\epsilon > 0$. Practically they should be seen
as $R^{\alpha+\epsilon}$ for arbitrary $0 < \epsilon \ls 1$. The estimates where such powers occur will obviously depend on $\epsilon$. 

By $B(x,R)$ we denote the ball centered at $x$ with radius $R$ in the underlying space (most of the time it will be $\R^n$ or $\R^{n+1}$). 

Let $\eta_0:\R^n \rightarrow [0,+\infty)$ be a Schwartz function, normalized in $L^1$, that is $\| \eta_0 \|_{L^1}=1$,
and with Fourier transform supported on the unit ball. Given some $r>0$ we denote by $\eta_r(x)=r^{-n} \eta_0(r^{-1} x)$
and note that $\hat \eta_r$ is supported in $B(0,r)$. 

A disk $D \subset \R^{n+1}$ has the form
\[
D=D(x_D,t_D;r_D)=\{(x,t_D) \in \R^{n+1}: |x-x_D| \leq r_D\},
\]
for some $(x_D,t_D) \in \R^{n+1}$ and $r_D > 0$. We define the associated smooth cut-off
\[
\tilde \chi_{D}(x,t)= (1+\frac{|x-x_D|}{r_D})^{-N}. 
\] 

A cube $Q \subset \R^{n+1}$ of size $R$ has the standard definition $Q=\{(x,t) \in \R^{n+1}: \|(x-x_Q,t-t_Q) \|_{l^\infty} \leq \frac{R}2 \}$, where $c_Q=(x_Q,t_Q)$ is the center of the cube. Given a constant $\alpha >0$ we define $\alpha Q$ to be the dilated by $\alpha$ of $Q$ from its center, that is
$\alpha Q=\{(x,t) \in \R^{n+1}: \|(x-x_Q,t-t_Q) \|_{l^\infty} \leq \alpha \cdot \frac{R}2 \}$. 

Given a cube $q \subset \R^{n+1}$ of size $r$ we will use three functions that are highly concentrated in $q$. 
The classic one is the characteristic function $\chi^c_q$ of $q$. We reserve the more standard notation $\chi_q$
for the next object which is used more frequently. This is build with the help of $\eta_0$ (we abuse
notation here as we should be using the corresponding $\eta_0:\R^{n+1} \rightarrow [0,+\infty)$ with similar properties):
\[
\chi_{q}(x) = \eta_0 (\frac{x-c(q)}r).
\]
This localization function has nice properties on the Fourier side. The other localization function is
\[
\tilde \chi_{q}(x) = (1+ |\frac{x-c(q)}r|)^{-N},
\]
for some large $N$. This localization has better properties on the physical side.

\subsection*{Acknowledgement}
Part of this work was supported by a grant from the Simons Foundation ($\# 359929$, Ioan Bejenaru).
Part of this work was supported by the National Science Foundation under grant No. $0932078000$ while the author was in residence at the 
Mathematical Research Sciences Institute in Berkeley, California, during the Fall 2015 semester. 

\section{Geometry of the surfaces and consequences}

We start by simplifying the setup. The surfaces are bounded, therefore we can always break them into smaller (and similar) pieces such as to accommodate the additional hypothesis described bellow. 

First we note that we can assume each $S_i$ to be of graph type: there is a smooth map 
$\varphi_i: U_i \subset \R^n \rightarrow \R$ such that $S=\{\Sigma_i(\xi)= (\xi,\varphi_i(\xi)): \xi \in U_i \}$.
Here $U_i$ are open, connected with compact closure. 
It is less important that the graphs are of type $\zeta_{n+1}=\varphi_i(\zeta_1,..,\zeta_{n})$ (we can have as well $\zeta_k=\varphi_i(\zeta_1,..,\bar \zeta_k,..\zeta_{n+1})$), although we can accommodate this by a rotation of coordinates. Then each flat leaf $S_{i,\alpha}$ corresponds to a flat leaf $U_{i,\alpha}$, in the sense that $\Sigma_i(U_{i,\alpha})=S_{i,\alpha}$; this is indeed the case since projections onto hyperplanes along a vector transversal to $S_i$ takes $2$-dimensional affine planes to $2$-dimensional affine planes. 

 We can find a system of coordinates $\bx_i: \R^n \rightarrow \R^n$ that parametrizes each leaf $U_{i,\alpha}$ into a new flat leaf 
$\tilde U_{i,\alpha}$ characterized by $\xi_3=constant,..,\xi_{n}=constant$. Finally, we assume that each $U_i$ has small enough diameter. 

Next, we derive a key geometric consequence of our setup. Given a surface $S_i$ we define $\calN_i=\{ N_i(\zeta_i): \zeta_i \in S_i \}$ be the set of normals at $S_i$. By $dspan \calN_i$ we denote the following subset of the classical span of $\calN_i$:
\[
dspan \calN_i = \{ \alpha N_\alpha + \beta N_\beta: N_\alpha, N_\beta \in \calN_i, \alpha, \beta \in \R \}.
\] 
$dspan \calN_i $ is the set of linear combinations of two vectors in $\calN_i$; it is not a linear subspace. 
With these notation in place, we claim the following result.
\begin{lema} \label{GL}
If $S_1, S_2, S_3$ satisfy the conditions i)-iii), then for any $N \in dspan \calN_1$ and any $N_2 \in \calN_2, N_3 \in \calN_3$
the following holds true for all real numbers $a,b,c$:
\begin{equation} \label{GLe}
|a N + b N_2 + c N_3| \ges \max(|a||N|,|b|,|c|).
\end{equation}
The statement is symmetric with respect to $S_1,S_2,S_3$. 
\end{lema}
We note that the occurrence of $|N|$ in $\max(|a||N|,|b|,|c|)$ is motivated by the fact that vectors in $dspan \calN_1$
are not normalized, but vectors in $\calN_2$ and $\calN_3$ are. 

\begin{proof} \eqref{GLe} is equivalent to a transversality condition between $N,N_2,N_3$:
\[
vol(N,N_2,N_3) \ges |N|
\]
We write $N=\alpha N_\alpha + \beta N_\beta$ for some $N_\alpha \ne N_\beta$ and consider $\gamma: [0,t_0] \rightarrow S_1$, a smooth curve with the property that 
$ N_1(\gamma(0))=N_\alpha$ and $ N_1(\gamma(t_0))=N_\beta$. We also assume that $|\gamma'(t)| = 1$ on $[0,t_0]$ and that $0 \leq t_0 \ll 1$; this is possible because we assumed $U_1$ of small diameter.  In addition, if $\alpha_0$ is such that $\gamma(0) \in S_{i,\alpha_0}$, we can assume that 
$\gamma'(0) \in (T_{\gamma(0)} S_{1,\alpha_0})^\perp$.  Then we have
\[
\begin{split}
N_1(\gamma(t_0)) & = N_1(\gamma(0)) + \int_{0}^{t_0} S_{N_1(\gamma(s))} \gamma'(s) ds \\
& = N_1(\gamma(0)) + t_0 S_{N_1(\gamma(0))} \gamma'(0) + O(t_0^2).  
\end{split}
\]
We then continue with
\[
\begin{split}
N & = \alpha N_1(\gamma(0))+ \beta (N_1 + t_0 S_{N_1(\gamma(0))} \gamma'(0) + O(t_0^2)  \\
& = (\alpha+\beta) N_1(\gamma(0)) + \beta t_0 S_{N_1(\gamma(0))} \gamma'(0) + \beta O(t_0^2) 
\end{split}
\]
The two vectors $N_1(\gamma(0))$ and $S_{N_1(\gamma(0))} \gamma'(0)$ are transversal,  thus 
$|N| \approx |\alpha+\beta| + t_0 |\beta|  |S_{N_1(\gamma(0))} \gamma'(0)|$ (here we use that $t_0 \ll 1$), and also 
\[
\begin{split}
vol(N,N_2,N_3) & \approx vol((\alpha+\beta) N_1(\gamma(0)) + \beta t_0 S_{N_1(\gamma(0))} \gamma'(0) ,N_2,N_3) \\
& \ges |(\alpha+\beta) N_1(\gamma(0)) + \beta t_0 S_{N_1(\gamma(0))} \gamma'(0)| \approx |N|
\end{split}
\]
where we have used the following consequence of \eqref{curva}:
\[
vol(N_1(\gamma(0)), S_{N_1(\gamma(0))} v ,N_2,N_3) \ges 1,
\]
which holds true for any unit vector $v \in (T_{\gamma(0)} S_{1,\alpha_0})^\perp \subset T_{\gamma(0)} S_{1}$.
\end{proof}

Using a similar argument as above, one can easily establish the following dispersive estimate
\begin{equation} \label{dN}
| N_i(\zeta_1) - N_i(\zeta_2) | \approx d(S_{i,\alpha_1}, S_{i,\alpha_2})
\end{equation}
where $S_{i,\alpha_1}, S_{i,\alpha_2}$ are the leafs to which $\zeta_1, \zeta_2$ belong to, respectively. Here the distance between 
$S_{i,\alpha_1}$ and $S_{i,\alpha_2}$ can be defined either by using geodesics inside the hypersurface $S_i$ (using the induced metric from the ambient space $R^{n+1}$) or, equivalently, by using the classical distance between sets in $R^{n+1}$.

\section{Restating the problem} \label{CT}
 
 \subsection{Rephrasing the problem in terms of free waves} \label{FW}
 
 We reformulate our problem in terms of free waves, this being motivated by the use of wave packets in order to prove Theorem \ref{mainT}. Once the wave packet decomposition is made and its properties of the packets are clear, the formalization of the problem as an evolution equation can be forgotten. 

Assume we are given a surface $S$ with a graph type parametrization $\zeta_{n+1}=\varphi(\xi)$ where $\xi = (\zeta_1,..,\zeta_n)$.
We rename the variable $\zeta_{n+1}$ by $\tau$, thus the equation of $S$ becomes $\tau=\varphi(\xi)$. We parametrize the physical space by $(x,t) \in \R^{n} \times \R$. We make the choice that $\tau$ is the Fourier variable corresponding to $t$, while
$\xi$ is the Fourier variable corresponding to $x$.  In what follows we use the convention that $\hat f$ denotes the Fourier transform of $f$ with respect to the $x$ variable.

We define the free wave $\phi = \calE f$ as follows
 \[
 \phi(x,t) = \calE f(x,t) = \int_{\R^n} e^{i(x\cdot \xi + t \varphi(\xi))} f(\xi) d\xi.
 \]
 Note that $\phi(0)=\check{f}$ and $\hat \phi(\xi,t) = e^{it \varphi(\xi)} \hat \phi(\xi,0)$. 
 We define the mass of a free wave by $M(\phi(t)):= \| \phi(t) \|^2_{L^2}$ and note that it is time independent:
 \[
 M(\phi(t)):= \| \phi(t) \|^2_{L^2}= \| \hat \phi(t) \|^2_{L^2} = \| \hat \phi(0) \|^2_{L^2}= \| \phi(0) \|^2_{L^2}= M(\phi(0)). 
 \]
The proof of \eqref{mainE} relies on estimating $\Pi_{i=1}^3 \calE_i f_i $ on cubes on the physical side and see how this 
behaves as the size of the cube goes to infinity by using an inductive type argument with respect to the size of the cube.
Before we formalize this strategy, we note that at every stage of the inductive argument we re-localize functions both on the physical and frequency space, and, as a consequence, we need to quantify the new support on the frequency side. This will be done by using the using the margin of a function. 

We assume we are given a reference set $V$ inside which we want to keep all functions supported. If $f$ is supported in $U \subset V$ we define
the margin of $f$ relative to $V$ by
\[
\mbox{margin}(f) := \mbox{dist}(\mbox{supp} (f), V^c). 
\]
In terms of free waves $\phi = \calE f$, the margin is defined by
\[
\mbox{margin}(\phi(t)) := \mbox{dist}(\mbox{supp}_\xi (\hat \phi(t)), V^c)= \mbox{dist}(\mbox{supp} (f), V^c),
\]
where we have used that the Fourier support of $\hat \phi(t)$ is time independent and that $\hat \phi(0)=f$. In other words, 
the margin of a free wave is time independent. 

In practice, we work with three different types of free waves, $\phi_i = \calE_i f_i$. They are assumed to be graphs with different phase functions
$\varphi_i$ and with potentially different ambient domain, that is $U_i$ are subsets of different subspaces isomorphic to $\R^n$ (for instance $U_i$ can be 
subsets of the hyperplanes $\xi_i=0$). The above construction changes only in making the choice of $\tau$ being the coordinate in the direction normal to the ambient hyperplane to which $U_i$ belongs to, while $\xi$ are the coordinates in the ambient hyperplane. Obviously, the margin of each $\phi_i$ is then defined with respect to some $V_i$ int the same ambient hyperplane. When choosing the reference sets $V_i$ we need to impose that the conditions i)-iii) hold true on $\Sigma_i(V_i)$ as well. 

Next, we prepare the elements that are needed for the induction on scale argument. 

\begin{defin} Let $p(3) \leq p \leq 1$. Given $R \geq C_0$ we define $A_p(R)$ to be the best constant for which the estimate
\begin{equation}
\| \Pi_{i=1}^k \phi_i   \|_{L^p(Q_R)} \leq A_p(R)  \Pi_{i=1}^k M(\phi_i)^\frac12. 
\end{equation}
holds true for all cubes $Q_R$ of size-length $R$, $\phi_i=\calE_i f_i$ and 
obeying the margin requirement
\begin{equation} \label{mrpg}
margin^i(\phi_i) \geq M-R^{-\frac14}, i=1,..,k.
\end{equation}
\end{defin}

The goal is to obtain an uniform estimate on $A_p(R)$ with respect to $R$. In the absence of the margin requirement above, 
$A_p(R)$ would be an increasing function. However, since the argument needs to tolerate 
the margin relaxation, we also define
\[
\bar A_p(R):= \sup_{1 \leq r \leq R} A_p(r) 
\]
and the new $\bar A_p(R)$ is obviously increasing with respect to $R$.  

Then \eqref{mainE}, and as a consequence the main result of this paper, Theorem \ref{mainT}, follow from the next result. 
\begin{prop} \label{keyP} Assume $0 < \epsilon < 1$. If $R \gg 2^{2C_0}$ and $R^{-\frac14+} \ll c \ll 1$, there exists $C(\epsilon)$
such that the following holds true:
\begin{equation} \label{ApR}
A_p(R) \leq (1+cC) \left( (1+cC)^p \left( \bar A_p(\frac{R}2) \right)^p   + 
\left( C(\epsilon) c^{-C} R^{ \frac{n+4}2(\frac1p-\frac32 \cdot \frac{n+2}{n+4})+\epsilon} \right)^p \right)^\frac1p.
\end{equation}
\end{prop}

We show how \eqref{mainE} follows from \eqref{ApR}. From the beginning we can assume the $\bar A_p(R) \geq 1$;
from this we obtain the simpler inequaltity
\[
A_p(R) \leq (1+cC)^2 \bar A_p(\frac{R}2) \left( 1 + \left( C(\epsilon) c^{-C} R^{ \frac{n+4}2(\frac1p-\frac32 \cdot \frac{n+2}{n+4})+\epsilon} \right)^p \right)^\frac1p.
\]
Since $p > \frac23 \cdot \frac{n+4}{n+2}$, we set $c^{-C}=R^{-\frac{n+4}4(\frac1p-\frac32 \cdot \frac{n+2}{n+4})-\epsilon}$, 
that is $c = R^{\frac{n+4}{4C}(\frac1p- \frac32 \cdot \frac{n+2}{n+4})+\frac{\epsilon}{C}}$, and note that $c$ satisfies $R^{-\frac14+} \ll c \ll 1$, provided $C(n,p)$ is large enough and $\epsilon$ is chosen small enough such that $\frac{n+4}{4}(\frac1p- \frac32 \cdot \frac{n+2}{n+4})+\epsilon < 0$. Then we apply the previous inequality to obtain
\[
A_p(R) \leq (1+cC)^2 \bar A_p(\frac{R}2) \left( 1 + 
\left( C(\epsilon) R^{ \frac{n+4}4(\frac1p-\frac32 \cdot \frac{n+2}{n+4})} \right)^p \right)^\frac1p.
\]
For $R$ large enough, $C(\epsilon) R^{ \frac{n+4}4(\frac1p-\frac32 \cdot \frac{n+2}{n+4})} < 1$, and by applying Bernoulli's inequality,
$(1+x)^r \leq 1+(2^r-1)x$ for $|x| \leq 1$ and $0 < r < 1$, we continue with
\[
A_p(R) \leq  (1+cC)^2  \bar A_p(\frac{R}2) \left( 1 + (2^\frac1p-1) C(\epsilon) R^{ \frac{n+4}4(\frac1p-\frac32 \cdot \frac{n+2}{n+4})} \right).
\]
Taking the maximum with respect to $r \in [\frac{R}2,R]$ gives
\[
\bar A_p(R) \leq (1+C R^{\frac{n+4}{4C}(\frac1p- \frac32 \cdot \frac{n+2}{n+4})})^2 \bar A_p(\frac{R}2) \left( 1 + (2^\frac1p-1) C(\epsilon) R^{ \frac{n+4}4(\frac1p-\frac32 \cdot \frac{n+2}{n+4})} \right).
\]
Since both powers of $R$ are negative, 
$\frac{n+3}{4C}(\frac1p-\frac{n+1}{n+3}), \frac{n+3}4(\frac1p-\frac{n+1}{n+3}) < 0$, this estimate can be iterated
to show that $\bar A_p(R)$ is uniformly bounded in terms of $\bar A_p(C^\alpha 2^{2C_0})$ for all $R \geq C^\alpha 2^{2C_0}$; 
here $\alpha$ is chose large enough so that if $R \geq C^\alpha 2^{2C_0}$, then $C(\epsilon) R^{ \frac{n+4}4(\frac1p-\frac32 \cdot \frac{n+2}{n+4})} < 1$. Since $\bar A_p(C^\alpha 2^{2C_0})$ is bounded by a constant depending on $C_0$, $C$ and $p$ (and $p$ influences the choice of $\epsilon$), \eqref{mainE} follows and we conclude the proof of Theorem \ref{mainT}. 
 
\subsection{Tables on cubes}
 
Let $Q \subset \R^{n+1}$ be a cube of radius $R$. Given $j \in \N$ we split $Q$ into $2^{(n+1)j}$ cubes of size $2^{-j} R$ and denote this family by $\calQ_j(Q)$; thus we have $Q=\cup_{q \in \calQ_j(Q)} q$.  
If $j \in \N$ and $0 \leq c \ll 1$ we define the $(c,j)$ interior $I^{c,j}(Q)$ of $Q$ by
\begin{equation} \label{icj}
I^{c,j}(Q) := \bigcup_{q \in \calQ_j(Q)} (1-c) q.
\end{equation}
Given $j \in \N$ we define a table $\Phi$ on $Q$ to be a vector $\Phi=(\Phi^{(q)})_{q \in \calQ_j(Q)}$ and define its mass by
\[
M(\Phi) = \sum_{q \in \calQ_j(Q)} M(\Phi^{(q)}). 
\]
We define the margin of a table as the minimum margin of its components:
\[
margin(\Phi) = \min_{q \in \calQ_j(Q)} margin(\Phi^{(q)}).
\]

Inspired by the Lemma 6.1 in \cite{Tao-BW}, we will make use of the following result. 
\begin{lema} \label{AVL}
Assume $0 < p < \infty$, $R \gg 1$, $0 < c \ll 1$ and $f$ smooth. Given a cube $Q_R \subset \R^{n+1}$ of size $R$,
there exists a cube $Q$ of size $2R$ contained in $4 Q_R$ such that
\begin{equation} \label{avrg}
\| f \|_{L^p(Q_R)} \leq (1+cC) \|  f\|_{L^p(I^{c,j}(Q))}
\end{equation}
\end{lema}
\begin{proof} Using Fubini's theorem, we have the following identity
\[
\int_{Q_R} \| f \|^p_{L^p((Q_R \cap I^{c,j}(Q(x,t;2R)))} dx dt= \int_{Q_R} |f(x,t)|^p |Q_R \cap I^{c,j}(Q(x,t;2R))| dx dt.
\]
From the definition of $I^{c,j}(Q(x,t;2R))$ it follows that
\[
|Q(x,t;2R) \setminus I^{c,j}(Q(x,t;2R))| \leq (n+1) c |Q(x,t;2R)| = (n+1) 2^{n+1} c |Q_R|
\]
and, as a consequence,
\[
|Q_R| \leq (1+(n+1)2^{n+1}c) | Q_R \cap  I^{c,j}(Q(x,t;2R)) |, \qquad \forall (x,t) \in Q_R.
\]
In the above we have used that if $(x,t) \in Q_R$ then $Q_R \subset Q(x,t;2R)$. 

Combining this estimates with the above identity, leads to
\[
\| f \|_{L^p}^p \leq \frac1{|Q_R|} \int_{Q_R} (1+(n+1)2^{n+1}c) \| f \|^p_{L^p((Q_R \cap I^{c,j}(Q(x,t;2R)))} dx dt
\]
By the pigeonholing principle, it follows that there is $(x,t) \in Q_R$ such that
\[
\| f \|^p_{L^p} \leq (1+(n+1)2^{n+1}c) \| f \|^p_{L^p((Q_R \cap I^{c,j}(Q(x,t;2R)))}
\]
and since $(1+(n+1)2^{n+1}c)^\frac1p \leq 1+cC $, the conclusion follows. Note that here the value of $p$ affects the choice of $C$.  
\end{proof}

\section{Wave packets} \label{SWP}

The wave packet construction is standard by now. The key elements are: an interpretation of the $\phi=\calE f$ 
as an evolution equation with initial data $\phi(0)=\check{f}$, a phase-spaces decomposition of the initial data in a linear space 
and an analysis of the evolution of each such component. 
The need of all these elements is the reason to require the more restrictive hypothesis of graph-like foliation. 

We now continue with the formalization of the wave packet construction for double-conical surfaces. We assume that $S$ is of double-conic type and it has the graph-type parametrization $\Sigma: U \rightarrow S$, where $\Sigma(\xi)=(\xi,\varphi(\xi))$ and with foliations
$U =  \cup_\alpha U_\alpha$, $S =  \cup_\alpha S_\alpha$, $\Sigma(U_\alpha)=S_\alpha$. 

For the foliation $U=  \cup_\alpha U_\alpha$, we choose a system of coordinates $\bf x \rm: U \rightarrow \tilde U$
such that for each leaf $U_{\alpha}$, the coordinates of $U_{\alpha}$ are $\xi_3=constant,..,\xi_{n}=constant$. Let
$\tilde U'=\pi(\tilde U)$, where $\pi : \R^{n} \rightarrow \R^{n-2}$ is the projection $\pi(\xi_1,..,\xi_n)= (\xi_3,..,\xi_n)$. 
Let $\tilde{\mathcal{L}}$ be a maximal $r^{-1}$-separated subset of $\tilde U' \subset \R^{n-2}$. 
For each $\tilde \xi \in \tilde{\mathcal{L}}$, $\bx^{-1}(\cdot, \cdot, \tilde \xi)$ is a leaf, that is
 $\bx^{-1}(\cdot, \cdot, \tilde \xi)=U_\alpha$ for some $\alpha$.  In each such leaf we pick $\xi_T$ and call $\mathcal{L}$ to be 
the set obtained this way. It is not important which $\xi_T \in \bx^{-1}(\cdot, \cdot, \tilde \xi)$ is chosen, since from condition ii) it follows that, for $\xi \in U_\alpha$, the normal $N(\Sigma(\xi))$ to $S$ is constant as $\xi$ varies inside the leaf $U_\alpha$. 
We denote by $U(\xi_T)$ the leaf $U_\alpha$ to which $\xi_T$ belongs and by $S(\xi_T)=\Sigma(U(\xi_T))$, the corresponding leaf on $S$. We note that $d(U(\xi_{T_1}), U(\xi_{T_2})) \approx d(\tilde \xi_1, \tilde \xi_2)$ which combined with \eqref{dN} gives
\begin{equation} \label{disp}
|N(\Sigma(\xi_{T_1})) - N(\Sigma(\xi_{T_2}))| \approx d(U(\xi_{T_1}), U(\xi_{T_2})) \approx d(\tilde \xi_1, \tilde \xi_2). 
\end{equation}

Let $L$ be the lattice $L=c^{-2} r \Z^n$. With $x_T \in L, \xi_T \in \mathcal{L}$ we define the tube $T=T(x_T,\xi_T):=\{ (x,t) \in \R^n \times \R: |x-x_T +  t \nabla \varphi(\xi_T) | \leq c^{-2} r \}$ and denote by $\calT$ the set of such tubes. One notices that $T$ is the $c^{-2}r$ neighborhood of the 
line passing through $(x_T,0)$ and direction $N(\Sigma(\xi_T))$. 

Associated to a tube $T \in \calT$, we define the cut-off $\tilde \chi_T$ on $\R^{n+1}$ by
\[
\tilde \chi_T(x,t)= \tilde \chi_{D(x_T -  t \nabla \varphi (\xi_T) ,t; c^{-2} r)}(x).
\] 
 We are ready to state the main result of this Section. 
\begin{lema} \label{LeWP} Let $Q$ be a cube of radius $R \gg 1$, let $c$ be such that $R^{-\frac14+} \ll c \les 1$ and let $J \in \N$ be such that $r =2^{-J} R \approx R^\frac12$. Let $\phi=\calE f$ be a free wave with $margin(\phi) > 0$. For each $T \in \calT$ there is a free solution 
$\phi _T$, that is localized in a neighborhood of size $CR^{-\frac12}$ of the leaf $S(\xi_T)$ and obeying $margin(\phi_T) \geq margin(f)-CR^{-\frac12}$. The map $ f \rightarrow \phi_T$ is linear and 
\begin{equation} \label{lind}
\phi= \sum_{T \in \calT} \phi_{T}.
\end{equation}
If $\mbox{dist}(T,Q) \geq 4 R$ then 
\begin{equation} \label{ld}
\| \phi_T \|_{L^\infty(Q)} \ls c^{-C} dist(T,Q)^{-N} M(\phi)^\frac12.
\end{equation}
The following estimates hold true
\begin{equation} \label{qest}
\sum_{T} \sup_{q \in Q_J(Q)} \tilde \chi_T(x_q,t_q)^{-N} \| \phi_T \|^2_{L^2(q)} \ls c^{-C} r M(\phi)  
\end{equation}
and
\begin{equation} \label{q00e}
\left( \sum_{q_0} M( \sum_T m_{q_0,T} \phi_T)  \right)^\frac12
\leq (1+cC) M(\phi),
\end{equation}
provided that the coefficients $m_{q_0,T} \geq 0$ satisfy
\begin{equation} \label{q0e}
\sum_{q_0} m_{q_0,T}=1, \qquad \forall T \in \calT.
\end{equation}
\end{lema}

The above construction was provided for the case of conic surfaces (single-conic in our context) in \cite{Tao-BW}, and for more general 
surfaces in \cite{Be2}. We sketch the main parts of the argument. We start with the  partition
\[
\tilde U'= \bigcup_{\xi \in \tilde{\mathcal{L}}} A_\xi
\]
where $A_\xi$ consists of the points in $\tilde U'$ that are closer to $\xi$ than any other elements of $\tilde{\mathcal{L}}$. 
Therefore $A_\xi$ belongs to the $O(r^{-1})$ neighborhood of $\xi$. 

Let $G$ be the set of all translations in $\R^{n-2}$ by vectors of size at most $O(r^{-1})$; in particular these translations
 differ from identity by $O(r^{-1})$. Let $d \Omega$
be a smooth compactly supported probability measure on the interior of $G$. 
For each $\Omega \in G$ and $\xi_0 \in \mathcal{L}$, we define the Fourier projectors by
\[
\mathcal{F}(P_{\Omega,\xi_0} g)(\xi) = \chi_{\Omega(A_{\pi \circ \bx(\xi_0)})} (\pi \circ \bx(\xi)) \hat g (\xi). 
\]
For fixed $\Omega \in G$, this leads to the decomposition:
\begin{equation} \label{Fdec}
g = \sum_{\xi_0 \in \mathcal{L}} P_{\Omega,\xi_0} g. 
\end{equation}
The terms above have good frequency support: $P_{\Omega,\xi_0} g$ is localized in frequency in a neighborhood of size 
$C r^{-1}$ of $U(\xi_0)$; indeed, this follows from the definition of the sets $A_\xi$, the properties of $\Omega$, the smoothness
properties of $\bx$ and the fact that $\pi \circ \bx (U(\xi_T))= \pi \circ \bx (\xi_T)$. 

We note the following: the sole reason for bringing in the additional coordinate map $\bx$ into the picture was to accommodate the simple group
structure $G$ used above; this is very easy in the context of the fibers $\bx(U_\alpha)$, or more precisely their projections $\pi \circ \bx(U_\alpha)$ as one simply uses the translations on $\R^{n-2}$. At the original level of the fibers $U_\alpha$, this would be more difficult, but probably doable. For instance, for the classical conical model \eqref{model}, one assigns to each leaf $U_\alpha$ an angle $\omega_\alpha \in \S^{n-2}$ and can use the rotations in $SO(n)$
that leave the $\S^{n-2}$ invariant to construct $G$. This is essentially the construction used in \cite{Tao-BW}. 

Now we proceed with the spatial localization. For each $x_0 \in L$, define
\[
\eta^{x_0}(x) = \eta_0 (\frac{c^2}r (x-x_0))
\]
and notice that, by the Poisson summation formula and properties of $\eta_0$, 
\begin{equation} \label{pois2}
\sum_{x_0 \in L} \eta^{x_0}=1.
\end{equation}
In Section \ref{FW} we interpreted $\phi=\calE f$ as an evolution with initial data $\phi(0)=\check f$. Based on this, we define 
\[
\phi_T(0)= \eta^{x_T}(x) \int P_{\Omega,\xi_T} \phi(0) d \Omega
\]
and evolve this, at all other times, by the free flow
\[
\phi_T(t)=e^{it\varphi(D)} \phi_T(0). 
\]
The statements \eqref{lind}-\eqref{qest} are standard in the wave packet theory, and they rely on the phase-space localization properties of the
initial data $\phi_T(0)$ and the properties of the flow $e^{it\varphi(D)}$; the properties of the flow $e^{it\varphi(D)}$ are closely related to the geometry 
of the characteristic surface $S$ and in our particular case the key property is that in the support of $\hat \phi_T(0)$ we have 
$| N(\Sigma(\xi))-N(\Sigma(\xi'))| \les r^{-1}$ or, in terms of phase function, $|\nabla \varphi(\xi)- \nabla \varphi(\xi')| \les r^{-1}$. This guarantees 
that the wave $\phi_T$ is highly localizes in the tube $T$. The details can be filled in using similar arguments to those in \cite[Lemma 4.1]{Be2} or 
\cite[Lemma 15.2]{Tao-BW}. 

Given that a statement of type \eqref{q00e} is less common in a wave packet decomposition,  we include a proof of it for convenience. 
This proof follows closely the one in \cite[Lemma 15.2]{Tao-BW}. By using the unitarity of $e^{it\varphi(D)}$ on $L^2$, \eqref{q00e} is reduced to the corresponding statement at time $t=0$. We denote $\phi(0)=g$ and the left hand-side of \eqref{q00e} becomes
\[
\left( \sum_{q_0} \| \int \sum_{\xi_0 \in \calL} \sum_{x_0 \in L} m_{q_0,T(x_0,\xi_0)} \eta^{x_0} P_{\Omega,\xi_0} g  d \Omega\|_{L^2}^2  \right)^\frac12
\]
 which is obviously bounded by
\begin{equation} \label{aux10}
\int \left( \sum_{q_0} \| \sum_{\xi_0 \in \calL} \sum_{x_0 \in L} m_{q_0,T(x_0,\xi_0)} \eta^{x_0} P_{\Omega,\xi_0} g  \|_{L^2}^2  \right)^\frac12 d \Omega.
\end{equation}
We define the set
\[
\tilde Y:= \bigcup_{\xi_0 \in \tilde \calL} \{ \xi \in A_{\xi_0}: d(\xi, A^c_{\xi_0}) > Cc^2 r^{-1}  \}
\]
and the associated multiplier $P_{\Omega(\tilde Y)}$ with symbol $\chi_{\Omega(\tilde Y)} \circ \pi \circ \bx (\xi)$, that is
\[
\mathcal{F} (P_{\Omega(\tilde Y)} h )= \chi_{\Omega(\tilde Y)} ( \pi ( \bx (\xi) )) \hat h(\xi). 
\]
Using the triangle inequality we bound the term \eqref{aux10} by
\begin{equation} \label{aux11}
\int \left( \sum_{q_0} \| \sum_{\xi_0 \in \calL} \sum_{x_0 \in L} m_{q_0,T(x_0,\xi_0)} \eta^{x_0}  P_{\Omega,\xi_0} P_{\Omega(\tilde Y)} g  \|_{L^2}^2  \right)^\frac12 d \Omega
\end{equation}
and 
\begin{equation} \label{aux12}
\int \left( \sum_{q_0} \| \sum_{\xi_0 \in \calL} \sum_{x_0 \in L} m_{q_0,T(x_0,\xi_0)} \eta^{x_0} P_{\Omega,\xi_0} (1-P_{\Omega(\tilde Y)}) g  \|_{L^2}^2  \right)^\frac12 d \Omega
\end{equation}
Given $\xi_1, \xi_2 \in \calL$, $P_{\Omega,\xi_1} P_{\Omega(\tilde Y)} g$ and $P_{\Omega,\xi_2} P_{\Omega(\tilde Y)} g$ have Fourier support in the sets 
$\bx^{-1}( \Omega(A_{\pi \bx(\xi_1)} \setminus \tilde Y))$ and  $\bx^{-1}( \Omega(A_{\pi \bx(\xi_2)} \setminus \tilde Y))$, respectively. 
The two sets $ \Omega(A_{\pi \bx(\xi_1)} \setminus \tilde Y)$ and $ \Omega(A_{\pi \bx(\xi_2)} \setminus \tilde Y)$ are at distance $\geq Cc^2 r^{-1}$ and the map $\bx^{-1}$ may change this distance by a factor, that is the distance between the sets $\bx^{-1}( \Omega(A_{\pi \bx(\xi_1)} \setminus \tilde Y))$ and  $\bx^{-1}( \Omega(A_{\pi \bx(\xi_2)} \setminus \tilde Y))$ is $\geq Cc^2 r^{-1}$, after redefining $C$. 

The functions $\eta^{x_0}$ have Fourier support in the set $|\xi| \leq c^2 r^{-1}$, thus the Fourier supports of 
$\eta^{x_1} P_{\Omega,\xi_1} P_{\Omega(\tilde Y)} g$ and $\eta^{x_2} P_{\Omega,\xi_2} P_{\Omega(\tilde Y)} g$ are also at distance 
$\geq Cc^2 r^{-1}$; recall that $\mathcal{F}(\eta^{x_1} P_{\Omega,\xi_1} P_{\Omega(\tilde Y)} g) = \hat \eta^{x_1} * \mathcal{F}(P_{\Omega,\xi_1} P_{\Omega(\tilde Y)} g)$. The conclusion to draw from this is that if $\xi_1 \ne \xi_2$, then $\eta^{x_1} P_{\Omega,\xi_1} P_{\Omega(\tilde Y)} g$ and $\eta^{x_2} P_{\Omega,\xi_2} P_{\Omega(\tilde Y)} g$ have disjoint Fourier support. As a consequence, the term in \eqref{aux11} equals
\[
\begin{split}
= & \int \left( \sum_{q_0} \sum_{\xi_0 \in \calL} \|  \sum_{x_0 \in L} m_{q_0,T(x_0,\xi_0)} \eta^{x_0}  P_{\Omega,\xi_0} P_{\Omega(\tilde Y)} g  \|_{L^2}^2  \right)^\frac12 d \Omega \\
= &  \int \left(  \sum_{\xi_0 \in \calL} \int |P_{\Omega,\xi_0} P_{\Omega(\tilde Y)} g (x)|^2  \sum_{q_0} \left( \sum_{x_0 \in L} m_{q_0,T(x_0,\xi_0)} \eta^{x_0} (x)\right)^2 dx   \right)^\frac12 d \Omega \\
\leq &  \int \left(  \sum_{\xi_0 \in \calL} \int |P_{\Omega,\xi_0} P_{\Omega(\tilde Y)} g (x)|^2  \left( \sum_{q_0}  \sum_{x_0 \in L} m_{q_0,T(x_0,\xi_0)} \eta^{x_0} (x)\right)^2 dx   \right)^\frac12 d \Omega \\
\end{split}
\]
 Using \eqref{q0e} and \eqref{pois2} we can continue with the 
 \[
 = \int \left(  \sum_{\xi_0 \in \calL} \int |P_{\Omega,\xi_0} P_{\Omega(\tilde Y)} g (x)|^2  dx   \right)^\frac12 d \Omega 
 \leq \int \| P_{\Omega(\tilde Y)} g \|_{L^2} d \Omega \leq \| g \|_{L^2},
 \]
  where in passing to the second inequality we have used the orthogonality of the projectors $P_{\Omega,\xi_0}$ with respect to $\xi_0$. 
 
 We use a similar argument for dealing with the term in \eqref{aux12}, except that now we invoke only almost orthogonality arguments, and not full orthogonality; therefore we can bound the term in \eqref{aux12} by
 \[
 \les \int \| (1-P_{\Omega(\tilde Y)}) g \|_{L^2} d \Omega.
 \]
  We use the Cauchy-Schwartz inequality and Plancherel to continue the sequence of bounds,
  \[
  \begin{split}
  \les & \left( \int \| (1-P_{\Omega(\tilde Y)}) g \|^2_{L^2} d \Omega \right)^\frac12 \\
  = &  \left( \int \| (1-\chi_{\Omega(\tilde Y)} ( \pi \bx (\xi))) \hat g(\xi) \|^2_{L^2} d \Omega \right)^\frac12 \\
  = &  \left( \int |\hat g(\xi)|^2 \int  (\chi_{\Omega(\tilde Y)^c}(\pi \bx(\xi)))^2 d \Omega d \xi \right)^\frac12 \\
  \end{split}
  \]
 For fixed $\xi$, the integral with respect to $\Omega$ is $\les c^2$, thus the final bound is 
 \[
 \les c \| g \|_{L^2}. 
 \]
 Combining the two bounds we derived for \eqref{aux11} and \eqref{aux12} leads to \eqref{q00e}. 
  
\section{Localization of the multilinear estimate} \label{LocME}
In \cite{Be1} we established a refinement of the generic trilinear restriction estimate under the additional assumption of small support of one of the terms involved, see \eqref{MEC} below. In this section we provide a slight modification of this result which is needed for technical reasons. 

We recall the setup from \cite{Be1}. We are given $3$ smooth hypersurfaces $S_i = \Sigma_i(U_i)$ with smooth parameterizations 
$\Sigma_i$. We assume that there exists $\nu >0$ such that
\begin{equation} \label{normal2}
vol (N_1(\zeta_1), N_{2}(\zeta_{2}), N_{3}(\zeta_{3})) \geq \nu,
\end{equation}
for all choices $\zeta_i \in \Sigma_i(U_i)$. Here by $vol (N_1(\zeta_1), N_{2}(\zeta_{2}), N_{3}(\zeta_{3}))$ we mean the volume of the $3$-dimensional parallelepiped spanned by the vectors $N_1(\zeta_1), N_{2}(\zeta_{2}), N_{3}(\zeta_{3})$. 

Assume that $\Sigma_1 (supp f_1) \subset  B(\calH,\mu)$, where $B(\calH,\mu)$ is the neighborhood of size $\mu$ of the $3$-dimensional affine subspace $\calH$. Assume that $|N_1(\zeta_1)- \pi_{\calH} N_1(\zeta_1)| \les \mu, \forall \zeta_1 \in \Sigma_1 (supp f_1)$,
where $\pi_{\calH} : \R^{n+1} \rightarrow \calH$ is the projection onto $\calH$. In addition assume that if $N_{i}, i=4,..,n+1$ is a basis of the normal space $\calH^\perp$ to $\calH$, then $N_1(\zeta_1),.., N_3(\zeta_3), N_4,..,N_{n+1}$ are transversal 
in the following sense:
\begin{equation} \label{normal}
| det(N_1(\zeta_1), N_{2}(\zeta_{2}), N_{3}(\zeta_{3}), N_{4},.., N_{n+1})| \geq \nu,
\end{equation}
for all choices $\zeta_i \in \Sigma_i(U_i)$. 

Under these hypothesis we proved that
\begin{equation} \label{MEC}
\| \calE_1  f_1  \calE_2 f_2  \calE_3 f_3 \|_{L^1(B(0,r))} \leq C(\epsilon) \mu^{\frac{n-2}2} r^\epsilon 
 \| f_{1} \|_{L^2(U_{1})} \| f_2 \|_{L^2(U_2)} \| f_3 \|_{L^2(U_3)}.
\end{equation}
In this section we show that this estimate localizes in the following sense: assume $q$ is a cube of size $r$ and that $\mu \ges r^{-1}$ then
\begin{equation} \label{MECL}
\begin{split}
 \| \calE_1  f_{1}  \calE_2 f_2  \calE_3 f_3  \|_{L^1(q)}
\leq  C(\epsilon) r^\epsilon \mu^{\frac{n-2}2} r^{-\frac32} \Pi_{i=1}^{3} \| \tilde \chi_q \calE_i f_i \|_{L^2}
\end{split}
\end{equation}
There are few aspects in the above estimate that need to be highlighted. Switching from balls of radius $r$ to cubes of radius $r$
makes no difference. The estimate says that inside $q$ the main contributions come from
$\calE_i f_i$ inside a dilate of $q$. This is simply a consequence of the finite speed of propagation.  The factor $r^{-\frac32} $ exhibits 
an apparent improvement over \eqref{MEC}. This is explained by the fact that while in \eqref{MEC}, the norms of $\calE_i f_i$
are measured along hyperplanes, in \eqref{MECL} the norms of $\calE_i f_i$ are measured on cubes of size $r$ and the additional dimension explains the additional factors. Finally, the condition $\mu \ges r^{-1}$ is crucial in using the localization on the physical side,
without altering the localization at scale $\les \mu$ of $\Sigma_1 (supp f_1)$. 

To prove \eqref{MECL}, we need to invoke the localization machinery developed in \cite{Be1}. For each $i=1,2,3$ we fix $\zeta_i^0 \in \Sigma_i(supp f_i)$, $N_i=N_i(\zeta_i^0)$ and let $\calH_i$ be the hyperplane on the physical side passing through the origin with normal $N_i$. We denote by $\pi_{N_i}$ the projection onto $\calH_i$ along $N_i$. Then, we choose a basis $N_i, i=4,..,n+1$ of the normal plane to $\calH$, let $\calH_i$ be the hyperplane on the physical side passing through the origin with normal $N_i$ and denote by $\pi_{N_i}$ the projection onto $\calH_i$ along $N_i$. The set $\{ N_i \}_{i=1,..,n+1}$ is a basis of $\R^{n+1}$.

In each $\calH_i, i=1,2,3$, we construct the set $\calC \calH_i(r)$ to be the set of $n$-dimensional cubes of size $r$. 
For $y_i \in \R$, we define $\calH_i + y_i N_i$ to be the translate of $\calH_i$ by $y_i N_i$ (in the normal direction).
Then $\calC \calH_i(r) + y_i N_i$ is the corresponding translate of $\calC \calH_i(r)$ by $y_i N_i$. 

For $i \in\{ 1,2,3 \}$, $r> 0$ and for a cube $q \in \calH_i$ of radius $r$, we define
$\chi_q: \calH_i \rightarrow \R$ by
\[
\chi_{q}(x) = \eta_0 (\frac{x-c(q)}r)
\]
Notice that $\mathcal{F}_i \chi_{q}$ has Fourier support in the ball of radius $\les r^{-1}$. This object is very similar to the ones defined in the Notation Section \ref{not}, the only difference is that they localize o cubes living in different linear spaces. By the Poisson summation formula and properties of $\eta_0$, 
\begin{equation} \label{pois}
\sum_{q \in \calC \calH_i(r)} \chi_{q}=1.
\end{equation}
Using the properties of $\chi_q$, a direct exercise shows that for each $N \in \N$, the following holds true
\begin{equation} \label{SN}
\sum_{q \in \calC \calH_i(r)}  \| \la \frac{x-c(q)}r \ra^{N} \chi_{q} g \|_{L^2}^2 \les_N \| g \|_{L^2}^2
\end{equation}
for any $g \in L^2(\calH_i)$. Here, the variable $x$ is the argument of $g$ and belongs to $\calH_i$.

Now we can start the argument of for \eqref{MECL}. Given any vector $y \in \R^{n+1}$ with $|y_i- c_i(q)| \leq r, i=1,2,3$ and $y_i=c_i(q), 4 \leq i \leq n+1$, we claim the following:
\begin{equation} \label{INS}
\begin{split}
& \| \Pi_{i=1}^{3} \calE_i f_i \|_{L^1(q)} \les C(\epsilon) r^\epsilon \mu^{\frac{n-2}2} \\
  &  \cdot \Pi_{i=1}^{3}  \left( \sum_{q' \in \calC \calH_i(r)+y_i N_i} \la \frac{d(\pi_{N_i} q,q')}r \ra^{-N} \| \la \frac{x-c(q')}r \ra^{N} \chi_{q'} \calE_i f_i \|_{L^2(\calH_i+y_i N_i)}^2 \right)^\frac12
\end{split}
\end{equation}
This allows us to average over the values of $(y_1,y_2,y_3)$ 
satisfying $|y_i- c_i(q)| \leq r$ (keeping $y_i=c_i(q), i \geq 4$) to further bound
\[
\begin{split}
& \| \Pi_{i=1}^{3} \calE_i f_i \|_{L^1(q)} \les C(\epsilon) r^\epsilon \mu^{\frac{n-2}2} r^{-\frac32} \\
  &  \cdot \Pi_{i=1}^{3}  \left( \int_{|y_i| \leq r} \sum_{q' \in \calC \calH_i(r)+y_i N_i} \la \frac{d(\pi_{N_i} q,q')}r \ra^{-N} \| \la \frac{x-c(q')}r \ra^{N} \chi_{q'} \calE_i f_i \|_{L^2(\calH_i+y_i N_i)}^2 \right)^\frac12 \\
  & \les  C(\epsilon) r^\epsilon \mu^{\frac{n-2}2} r^{-\frac32} \Pi_{i=1}^{3} \| \tilde \chi_q \calE_i f_i \|_{L^2}
\end{split}
\]
This gives the estimate \eqref{MECL}. All that is left is the justification of \eqref{INS}. We note that this type of a posteriori  
improvement of the multilinear estimate was at the core of the arguments in \cite{Be1} and that, while \eqref{INS} is not explicitly 
established there, its proof is essentially derived along the same lines as $(3.6)$ in that paper. 

We first note that, without restricting the generality of the argument we can assume that $c(q)=0$. Following the argument below, it will
be clear that the proof for $y=0$ extends to the more general case when $|y_i| \leq r, i=1,2,3$.

We fix $i=1$. In \cite{Be1} we explained that the problem is reducible to the case when $\Sigma_1:U_1 \subset \calH_1 \rightarrow \R^{n+1}$ with $\Sigma_1(\xi') =(\xi', \varphi_1(\xi'))$ and  $supp f_1 \subset B(\calH \cap \calH_1, \mu)$,
the $\mu$ neighborhood of $\calH \cap \calH_1 \subset \calH_1$. Then we have the representation
\begin{equation} \label{E1}
\calE_1 f (x_1,x')= \int_{U_1} e^{i (x' \xi' + x_1 \varphi_1(\xi'))} f(\xi') d\xi'.
\end{equation}
 We highlight a commutator estimate which is needed due to the uncertainty principle. 
For any fixed $x_0 \in \R^{n+1}$, it holds true that
\begin{equation} \label{com}
(x'-x'_0-x_1 \nabla \varphi_1(\frac{D'}i))^N \calE_1 f_1 =  \calE_1 (\mathcal{F}_1( (x'-x'_0)^N  \mathcal{F}_1^{-1} f_1)), \quad \forall N \in \N. 
\end{equation}
This is a direct computation using \eqref{E1} and it suffices to check it for $N=1$. 

Fix $q' \in \calC \calH_1(r)+y$. With the notation $A=C(\epsilon) \mu^{\frac{n-2}2} r^\epsilon$ we have
\[
\begin{split}
& \| (x'- c(q')-x_1 \nabla \varphi_1(\xi'_0)) \calE_1 \mathcal{F}_1 ( \chi_{q'}  \mathcal{F}_1^{-1} f_1) \cdot \Pi_{i=2}^{3}  \calE_i f_i \|_{L^{1}(q)} \\
= & \| (x'- c(q')-x_1 \nabla \varphi_1(\xi')) \calE_1 \mathcal{F}_1 ( \chi_{q'} \mathcal{F}_1^{-1} f_1) \cdot \Pi_{i=2}^{3}  \calE_i f_i  \|_{L^1(q)} \\
+ & \| x_1 ( \nabla \varphi_1(\xi'_0) - \nabla \varphi_1(\xi')) \calE_1 \mathcal{F}_1 ( \chi_{q'} \mathcal{F}_1^{-1} f_1) \cdot \Pi_{i=2}^{3}  \calE_i f_i  \|_{L^1(q)} \\
= & \|  \calE_1 \mathcal{F}_1 ( (x'-c(q'))  \chi_{q'} \mathcal{F}_1^{-1} f_1) \cdot \Pi_{i=2}^{3}  \calE_i f_i  \|_{L^1(q)} \\
+ & \| x_1  \calE_1 \mathcal{F}_1 ( ( \nabla \varphi_1(\xi'_0) - \nabla \varphi_1(\xi')) \chi_{q'} \mathcal{F}_1^{-1} f_1) \cdot \Pi_{i=2}^{3}  \calE_i f_i  \|_{L^1(q)} \\
\leq & A \left(  \| (x'-c(q')) \chi_{q'} \mathcal{F}_1^{-1} f_1 \|_{L^2} +  r \| (\nabla \varphi_1(\xi'_0) - \nabla \varphi_1(\xi')) \chi_{q'} \mathcal{F}_1^{-1} f_1 \|_{L^2} \right) \Pi_{i=2}^{3} \| f_i \|_{L^2}  \\
\les & A \left(  \| (x'-c(q')) \chi_{q'} \mathcal{F}_1^{-1} f_1 \|_{L^2} + r \| \chi_{q'} \mathcal{F}_1^{-1} f_1 \|_{L^2} \right)  \Pi_{i=2}^{3} \| f_i \|_{L^2} \\
\les & r A \| \la \frac{x'-c(q')}{r} \ra \chi_{q'} \mathcal{F}_1^{-1} f_1 \|_{L^2}  \Pi_{i=2}^{3} \| f_i \|_{L^2}
\end{split}
\]
We have used the following: \eqref{com} in justifying the equality between the terms on the second and fourth line, the induction hypothesis and the fact that inside $Q$ we have $|x_1| \les r$ to justify the inequality in the sixth line 
(note that this part extends easily for more general $y_1$ provided $|y_1| \leq r$). Since $\chi_{q'}$ has Fourier support in a ball of radius $\les r^{-1} \leq \mu$, $\calE_1 \mathcal{F}_1 ( (x'-c(q'))  \chi^0_{q'} \mathcal{F}_1^{-1} f_1)$ keeps the localization property of $f_1$, that is 
$supp \mathcal{F}_1 ( (x'-c(q'))  \chi^0_{q'} \mathcal{F}_1^{-1} f_1) \subset B(\calH \cap \calH_1, C \mu)$ for some fixed constant $C$.

For $x' \in \pi_{N_1}(q)$, it holds that 
$\la \frac{x'- c(q')-x_1 \nabla \varphi_1(\xi'_0)}r \ra \approx \la \frac{d(\pi_{N_1}(q),q')}r \ra$. 
This is justified by the fact that $|x_1| \les r$ and $|\nabla \varphi_1(\xi'_0)| \leq 1$, 
therefore the contribution of $|x_1 \nabla \varphi_1(\xi'_0)| \les r$ is negligible. From this and the previous set of estimates,
we conclude that
\[
 \| \calE_1 \mathcal{F}_1 ( \chi_{q'} \mathcal{F}_1^{-1} f_1) \cdot \Pi_{i=2}^{3}  \calE_i f_i \|_{L^1(q)} \les A
  \la  \frac{d(\pi_{N_1}q,q')}{r} \ra^{-1} \| \la \frac{x'-c(q')}{r} \ra \chi_{q'} \mathcal{F}_1^{-1} f_1 \|_{L^2}  \Pi_{i=2}^{3} \| f_i \|_{L^2}
\]
Repeating the argument gives
\[
 \| \calE_1 \mathcal{F}_1 ( \chi_{q'} \mathcal{F}_1^{-1} f_1) \cdot \Pi_{i=2}^{3}  \calE_i f_i \|_{L^1(q)} \les_N A
  \la  \frac{d(\pi_{N_1}q,q')}{r} \ra^{-N} \| \la \frac{x'-c(q')}{r} \ra^N \chi_{q'} \mathcal{F}_1^{-1} f_1 \|_{L^2}  \Pi_{i=2}^{3} \| f_i \|_{L^2}
\]

Using \eqref{pois} and the above, we obtain
\[
\begin{split}
& \| \calE_1 f_1 \cdot \Pi_{i=2}^{n+1}  \calE_i f_i \|_{L^1(q)} \\
 \leq  & A
\sum_{q' \in \calC \calH_1(r)} \| \calE_1 \mathcal{F}_1 ( \chi_{q'} \mathcal{F}_1^{-1} f_1) \cdot \Pi_{i=2}^{3}  \calE_i f_i \|_{L^1(q)} \\
 \les_N & A \left( \sum_{q' \in \calC \calH_{1}(r)} \la  \frac{d(\pi_{N_1}q,q')}{r} \ra^{-N} \| \la \frac{x'-c(q')}{r} \ra^N \chi_{q'} \mathcal{F}_1^{-1} f_1 \|_{L^2} \right)  \Pi_{i=2}^{3} \| f_i \|_{L^2} \\
 \les_N & A \left( \sum_{q' \in \calC \calH_{1}(r)} \la  \frac{d(\pi_{N_1}q,q')}{r} \ra^{-(2N-4)} \| \la \frac{x'-c(q')}{r} \ra^N \chi_{q'} \mathcal{F}_1^{-1} f_1 \|_{L^2}^2\right)^\frac12  \Pi_{i=2}^{3} \| f_i \|_{L^2}.
\end{split}
\]
This is the improvement we claimed in \eqref{INS} for $\calE_1 f_1$. The improvements for the $\calE_2 f_2$ and $\calE_3 f_3$
are obtained in a similar manner.

\section{Table construction and the induction argument} \label{SI}

This section contains the main argument for the proof of Theorem \ref{mainT}. In Proposition \ref{NLP}
we construct tables on cubes: this is a way of re-organizing the information on one term, say $\phi_1$, at smaller scales
based on information from one of the other interacting terms, $\phi_2$ or $\phi_3$. This type of argument is inspired by the work on the conic surfaces of Tao in \cite{Tao-BW}. Based on this table construction, we will prove the inductive bound claimed in Proposition \ref{keyP}. 

\begin{prop} \label{NLP}
Let $Q$ be a cube of size $R \gg 2^{2C_0}$.
Assume $\phi_i=\calE_i f_i, i \in \{1,2,3\}$ have positive margin.  
Then there is a table $\Phi_1=\Phi_c(\phi_1, \phi_2,Q)$ with depth $C_0$
such that the following properties hold true: 
\begin{equation} \label{dec}
\phi_1= \sum_{q \in  \calQ_{C_0}(Q)} \Phi_1^{(q)}, 
\end{equation}
\begin{equation} \label{MPhi2}
margin(\Phi) \geq margin(\phi)-C R^{-\frac12}.
\end{equation}

\begin{equation} \label{PhiM2}
M(\Phi) \leq (1+cC) M(\phi),
\end{equation}
and for any $q',q'' \in \calQ_{C_0}(Q), q'\ne q''$
\begin{equation} \label{Phia2}
\| \Phi_1^{(q')} \phi_2 \phi_3 \|_{L^1((1-c)q'')} \ls_\epsilon c^{-C} R^{-\frac{n-2}4+\epsilon}  \Pi_{i=1}^3 M^\frac12(\phi_i).
\end{equation}

\end{prop}

\begin{rem} \label{RNLP}
The above result is stated for scalar $\phi_1,\phi_2,\phi_3$, but it holds for vector versions as well. Most important 
is that we can construct $\Phi_1=\Phi_c(\phi_1, \Phi_2, Q)$ where $\Phi_2$ is a vector free wave and all its scalar components 
satisfy similar properties to the $\phi_2$ above. 
\end{rem}

\begin{rem} \label{RNLP1}
We note that $\Phi_1=\Phi_c(\phi_1, \phi_2,Q)$ means that the table  $\Phi_1$ is constructed from $\phi_1$, which is natural
in light of \eqref{dec}, and $\phi_2$. But it does not depend on $\phi_3$. Obviously, we could have constructed it from $\phi_1$
and $\phi_3$, ending with a different object. 
\end{rem}

\begin{proof}  There are several scales involved in this argument. The large scale is the size $R$ of the cube
$Q$. The coarse scale is $2^{-C_0} R \gg R^\frac12$, this being the size of the smaller cubes
in $\calQ_{C_0}(Q)$ and the subject of the claims in the Proposition. Then there is the fine scale $r=2^{-j}R$ chosen such that
 $r \approx R^\frac12$. Notice that $r$ is the proper scale for wave packets corresponding to time scales $R$ and also
 that their scale is $c^{-2} r \ll 2^{-C_0} R$, last one being the scale of cubes in $\calQ_{C_0}(Q)$. 
 
We use Lemma \ref{LeWP} with $J=j$ to construct the wave packet decomposition for $\phi_1$,
\[
\phi_1 = \sum_{T_1 \in \calT_1} \phi_{1,T_1}. 
\]
 For any $q_0 \in \calQ_{C_0}(Q)$ and $T_1 \in \calT_1$ we define
\[
m_{q_0,T_1}:= \|  \tilde{\chi}_{T_1} \phi_2 \|^2_{L^2(q_0)}
\]
and 
\[
m_{T_1}:= \sum_{q_0 \in \calQ_{C_0}(Q)} m_{q_0, T_1}.
\]
Based on this we define
\begin{equation} \label{dPq0}
\Phi_1^{(q_0)}:= \sum_{T_1 \in \calT_1} \frac{m_{q_0,T}}{m_{T_1}} \phi_{1,T_1}. 
\end{equation}
By combing the definitions above with the decomposition property \eqref{lind}, we obtain 
\[
\phi_1= \sum_{q_0 \in \calQ_{C_0}(Q)} \Phi_1^{(q_0)}
\]
thus justifying \eqref{dec}. 

The margin estimate \eqref{MPhi2} follows from the margin estimate on tubes provided by Lemma \ref{LeWP}. 
The coefficients $m_{q_0,T}$ satisfy \eqref{q0e}, thus the estimate \eqref{PhiM2} follows from 
 \eqref{q00e}. 

All that is left to prove is \eqref{Phia2}, which is equivalent to 
\begin{equation} \label{red1}
 \sum_{q \in \calQ_j(Q):d(q,q_0) \ges cR} \| \Phi^{(q_0)}_1  \phi_2 \phi_3 \|_{L^1(q)} 
\ls_\epsilon c^{-C} r^{-\frac{n-2}2+\epsilon} \Pi_{i=1}^3 M(\phi_i).
\end{equation}

Note that the cubes $q$ are selected at the finer scale dictated the size of cubes in $\calQ_j(Q)$. From the definition of $\Phi^{(q_0)}$ in 
\eqref{dPq0} we can discard the tubes $q$ which do not intersect $4Q$ based on \eqref{ld}, in the sense that their contribution 
to \eqref{red1} will give a better estimate. 

For the tubes intersecting $4Q$, we make another simplification motivated  by \eqref{qest} 
and focus on the tubes which intersect $q$, that is we focus on the following term 
\[
 \sum_{q \in \calQ_j(Q):d(q,q_0) \ges cR}   \|  \sum_{T_1 \cap q \ne \emptyset}  \frac{m_{q_0,T_1}}{m_{T_1}} \phi_{1,T_1} 
 \phi_2 \phi_3 \|_{L^1(q)} 
\]
Essentially \eqref{qest}  says that the other tubes have off-diagonal type contribution, that is there are enough gains in the case 
$T_1 \cap q = \emptyset$ to perform any summation. Note that in order to keep the notation simple, we skipped writing 
$T_1 \in \calT_1$ and will do so throughout the rest of this proof. 

Using the localized form of the trilinear estimate \eqref{MECL}, we obtain
\[
\| \Phi^{(q_0)}_1   \phi_2 \phi_3  \|_{L^1(q)}  \ls_\epsilon r^{-\frac{n+1}2+\epsilon}  \sum_{T_1} \frac{m_{q_0,T_1}}{m_{T_1}} \| \phi_{1,T_1} \tilde \chi_q \|_{L^2}  \| \phi_2 \tilde \chi_q \|_{L^2} \| \phi_3 \tilde \chi_q \|_{L^2} .
\]
Here we used that $\phi_{1,T_1}$ is a free wave of type $\calE_1 g_1$ (for some $g_1$) that is localized in frequency
 in the $\mu=r^{-1}$ neighborhood of the $2$-dimensional affine plane containing the leaf $S_{1, \alpha}$ where we choose the leaf  such that $\Sigma_1(\xi_T) \in S_{1, \alpha}$. If we add to the leaf the normal directions $N_1(\Sigma(\xi_T))$, we obtain a $3$-dimensional affine subspace $\calH$ with the properties required to invoke \eqref{MECL}.

Using the obvious inequality $\frac{m_{q_0,T_1}}{m_{T_1}} \leq \frac{m^\frac12_{q_0,T_1}}{m^\frac12_{T_1}}$, we obtain:
\[
\begin{split}
&  \sum_{ T_1 \cap q \ne \emptyset}  \frac{m_{q_0,T_1}}{m_{T_1}} \| \phi_{1,T_1} \tilde \chi_q \|_{L^2}\\
 \ls & 
\left( \sum_{T_1} \frac{\|\phi_{1,T_1} \tilde\chi_q \|^2_{L^2}}{m_{T_1} \tilde \chi_{T_1}(x_q,t_q)} \right)^\frac12 
 \left(  \sum_{T_1}  m_{q_0,T} \tilde \chi_{T_1}(x_q,t_q)  \right)^\frac12.
\end{split}
\]
Next we claim the following estimate
\begin{equation} \label{keyw}
\begin{split}
\sum_{T_1 \in \calT_1} m_{q_0,T_1} \tilde \chi_{T_1}(x_q,t_q) & \ls  \| \tilde \chi_{S(q)} \phi_2 \|_{L^2}^2. 
\end{split}
\end{equation}
Using the definition of $m_{q_0,T_1}$ we identify  the function 
\[
\tilde \chi_{S(q)}= ( \sum_{T_1 \in \calT_1} \tilde \chi(x_q,t_q) \tilde \chi_{T_1} ) \chi_{q_0}
\]
which makes  \eqref{keyw} hold true. 
Then we note that $\tilde \chi_{S(q)}$ has the following decay property:
\[
\tilde \chi_{S(q)} (x,t) \ls c^{-4} \left( 1+ \frac{d((x,t), S(q))}{c^{-2}r}\right)^{-N}.
\]
Here the surface $S(q)$ is the translate by $c(q)$ of the neighborhood of size $r$ of cone of normals at $S_1$, which we denote by $\mathcal{CN}_1:=\{\alpha N_1(\zeta), \zeta \in S_1, \alpha \in \R \}$.  It is important to note that we do not consider the whole cone but only the part with $cR \leq \alpha \ls R$. Then the estimate is a consequence of the fact that the tubes $T_1$ passing thorough $q$ separate inside $q_0$ and of the separation between $q$ and $q_0$, which is quantified by $d(q,q_0) \ges cR$. Quantitatively speaking, given a point in $q_0$ close to $S(q)$, there are $\ls c^{-4}$ tubes $T_1$ passing through the point and $q$ - this follows from the dispersion estimate \eqref{disp} and 
the geometry of the family of tubes $\calT_1$.

Next we claim the following estimate:
\begin{equation} \label{estau1}
\sum_q \| \tilde \chi_{S(q)} \phi_2 \|_{L^2}  \left( \sum_{T_1 \cap q \ne \emptyset} \frac{\|\phi_{1,T_1}  \tilde \chi_q \|^2_{L^2}}{m_{T_1} \tilde \chi_{T_1}(x_q,t_q)} \right)^\frac12
  \| \phi_2 \tilde \chi_q \|_{L^2}  \| \phi_3 \tilde \chi_q \|_{L^2} 
  \ls c^{-C} r^{\frac{3}2} \Pi_{i=1}^3 M(\phi_i).
\end{equation}
Combing \eqref{estau1} with \eqref{keyw} gives \eqref{red1} and this concludes the proofs of all claims of the Proposition.

All that is left to prove is \eqref{estau1} and this can be broken down into the following two claims:
\begin{equation} \label{mes}
\sum_q \| \phi_2  \tilde \chi_{S(q)} \|_{L^2}^2  \| \phi_3 \tilde \chi_q \|_{L^2}^2 \ls c^{-C} r^{2}  M(\phi_2)  M(\phi_3)
\end{equation}
and
\begin{equation} \label{mes2}
\sum_q \left( \sum_{T_1 \cap q \ne \emptyset} \frac{\|\phi_{1,T_1}  \tilde \chi_q \|^2_{L^2}}{m_{T_1} \tilde \chi_{T_1}(x_q,t_q)} \right)
 \| \phi_2 \tilde \chi_q \|^2_{L^2} \ls c^{-C} r M(\phi_1).
\end{equation}
The proof of \eqref{mes2} is similar to the one we used in the bilinear theory, see the proof of the corresponding theorem \cite{Be2}. 
By rearranging the sum, it suffices to show
\[
\sum_{T_1} \sum_{q \cap T_1 \ne \emptyset}   \frac{\|\phi_{1,T_1} \tilde \chi_q \|^2_{L^2} 
\| \phi_2 \tilde \chi_q \|^2_{L^2}}{m_{T_1} \tilde \chi_{T_1}(x_q,t_q)} \ls r M(\phi_1).
\]
The inner sum is estimated as follows:
\[
\sum_{q \cap T_1 \ne \emptyset}   \frac{ \| \phi_2 \tilde \chi_q \|^2_{L^2}}{m_{T_1} \tilde \chi_{T_1}(x_q,t_q)}
\ls   \frac{ \| \phi_2 \tilde \chi_{T_1} \|^2_{L^2}}{m_{T_1}} \ls 1,
\]
and the outer one is estimated by
\[
\sum_{T_1} \sup_{q} \|\phi_{1,T_1} \tilde \chi_q \|^2_{L^2}  \ls r \sum_{T_1} M(\phi_{1,T_1}) \ls r M(\phi_1),
\]
which is obvious given the size of $q$ in the $x_1$-direction is $\approx r$ and the mass of $\phi_{1,T_1}$
is constant across slices in space with $x_1=constant$. 

We continue with the proof of \eqref{mes}. Using the fast decay of $\tilde \chi_{S(q)}$ away from $S(q)$ and of 
 $\tilde \chi_q$ away from $q$, it suffices to show that
\begin{equation} \label{mes3}
\sum_q \| \chi^c_{S(q)} \phi_2 \|_{L^2}^2  \|  \chi^c_q \phi_3 \|_{L^2}^2 \ls r^{2}  M(\phi_2)  M(\phi_3),
\end{equation}
where by $\chi^c_{S(q)}$ we denote the characteristic function of the set $S(q)$ and recall that $\chi^c_q$ is the characteristic function of $q$; this will come at the cost of picking 
a factor of $c^{-C}$. 

We define the following relation: $q' \sim q$ if $q' \cap S(q) \ne \emptyset$ and note that this is equivalent to saying that there is a tube $T_1 \in \calT_1$ intersecting both $q$ and $q'$ and that $d(q,q') \ges cR$. We start from the obvious inequality
\[
 \| \chi_{S(q)}^c \phi_2 \|_{L^2}^2 \ls \sum_{q' \sim q}  \| \chi_{q'}^c \phi_2 \|_{L^2}^2
\]
which implies
\[
\sum_q \| \chi_{S(q)}^c \phi_2 \|_{L^2}^2  \|  \chi_q^c \phi_3  \|_{L^2}^2 \ls 
\sum_q \sum_{q' \sim q} \| \chi_{q'}^c \phi_2 \|_{L^2}^2 \|  \chi_q^c \phi_3 \|_{L^2}^2
\]
We are tacitly using again at this point the full dispersion property of the set of normals $\calN_1$: the tubes $T_1$ passing through $q$
separate inside $q_0$; in the absence of this property, the above inequality would fail, as we would encounter large number of tubes 
$T_1 \in \calT_1$ passing through both $q$ and $q'$. 

At this point we use a wave packet decomposition for $\phi_2$ and $\phi_3$: we invoke again Proposition \ref{LeWP}, but this time with $c \approx 1$. Using \eqref{qest}, we reduce \eqref{mes3} to the following 
\begin{equation} \label{final}
\sum_q \sum_{q' \sim q} \left( \sum_{T_2 \cap q' \ne \emptyset} M(\phi_{T_2}) \right) 
 \left( \sum_{T_3 \cap q \ne \emptyset} M(\phi_{T_3}) \right) \ls M(\phi_2)  M(\phi_3).
\end{equation}
Here we skip writing down that the sum run over $T_2 \in \calT_2$ and $T_3 \in \calT_3$ for keeping the notation shorter and we will do so for the rest of the argument.

The key point in justifying \eqref{final} is that, as we vary $(q,q')$ with $q' \sim q$, the number of occurrences of a pair of tubes $(T_2,T_3)$  on the left hand-side is bounded by a universal constant. Indeed, if that is the case we use \eqref{q00e} to obtain the bound
\[
\sum_q \sum_{q' \sim q} \left( \sum_{T_2 \cap q' \ne \emptyset} M(\phi_{T_2}) \right) 
 \left( \sum_{T_3 \cap q \ne \emptyset} M(\phi_{T_3}) \right) \les \sum_{T_2 \in \calT_2} \sum_{T_3 \in \calT_3}
  M(\phi_{T_2})  M(\phi_{T_3}) \ls M(\phi_2)  M(\phi_3).
\]

We finish the argument by establishing an upper bound on the number of occurrences of a pair of tubes $(T_2,T_3)$ 
on the left-hand side of \eqref{final}. Assume that a pair $(T_2,T_3)$ shows up multiple times. That means that there are $(q,q'), (\tilde q, \tilde q')$ such that $q \sim q', \tilde q \sim \tilde q'$ and $q' \cap T_2 \ne \emptyset, \tilde q' \cap T_2 \ne \emptyset$,
$q \cap T_3 \ne \emptyset, \tilde q \cap T_3 \ne \emptyset$. We tolerate repeated occurrences coming from the setup 
$d(q, \tilde q) \les r$ and $d(q', \tilde q') \les r$, which are bounded by a universal constant, but rule out all the others.

Consider the case $d(q, \tilde q), d(q', \tilde q') \gg r$. This implies the following:

$c(q)-c(\tilde q) = \alpha_3 N_3 + O(r)$ for some $N_3=N_3(\zeta_3), \zeta_3 \in S_3$, $|\alpha_3 | \gg r$, 

$c(q')-c(\tilde q') = \alpha_2 N_2 + O(r)$ for some $N_2=N_2(\zeta_2), \zeta_2 \in S_2$, $|\alpha_2 | \gg r$,

$c(q)-c(q')=\alpha_1 N_1 + O(r), c(\tilde q)-c(\tilde q')= \tilde \alpha_1 \tilde N_1 + O(r), N_1, \tilde N_1 \in \mathcal{CN}$
$|\alpha_1|, |\tilde \alpha_1| \gg r$. 

Since $c(q)-c(q') - (c(\tilde q)-c(\tilde q'))= c(q)-c(\tilde q) - (c(q')-c(\tilde q'))$, this implies 
\[
\alpha_1 N_1 - \tilde \alpha_1 \tilde N_1 =  \alpha_3 N_3 - \alpha_2 N_2 + O(r)
\]
Since $\alpha_1 N_1 - \tilde \alpha_1 \tilde N_1 \in dspan \calN_1$ we can invoke the result in Lemma \ref{GL} to obtain
\[
| \alpha_1 N_1 - \tilde \alpha_1 \tilde N_1 -  \alpha_3 N_3 + \alpha_2 N_2 | \geq max(|\alpha_2|, |\alpha_3|) \gg r
\]
which is in contradiction with the previous statement. 

The other two cases, $d(q, \tilde q) \gg r, d(q', \tilde q') \les r$ and $d(q, \tilde q) \les r, d(q', \tilde q') \gg r$ are ruled out in a similar way: one notices that in the above proof we needed only one of coefficients $\alpha_2, \alpha_3$ to have large absolute value.  

\end{proof}

\begin{proof}[Proof of Proposition \ref{keyP}] Let $\phi_i=\calE_i f_i$ satisfying the margin requirements \eqref{mrpg}. Let $Q_R$ be an arbitrary cube of radius $R$. From Lemma \ref{AVL}
it follows that there is a cube $Q \subset 4Q_R$ of size $2R$ such that
\begin{equation} \label{aux21}
\| \phi_1 \phi_2 \phi_3 \|_{L^p(Q_R)} \leq (1+cC) \| \phi_1 \phi_2 \phi_3 \|_{L^p(I^{c,j}(Q))}.
\end{equation}
Using the result of Proposition \ref{NLP} we build table $\Phi_1=\Phi_c(\phi_1,\phi_2,Q)$ on $\phi_1$ with depth $C_0$ and estimate as follows
\begin{equation*}
\begin{split}
\| \phi_1 \phi_2 \phi_3 \|^p_{L^p(I^{c,C_0}(Q))} & = \sum_{q_0' \in \calQ_{C_0}(Q)} \| \phi_1 \phi_2 \phi_3 \|^p_{L^p((1-c)q_0')} \\
& \leq \sum_{q_0,q_0' \in \calQ_{C_0}(Q)} \| \Phi^{(q_0)}_1 \phi_2 \phi_3 \|^p_{L^p((1-c)q_0')} \\ 
&  = \sum_{q_0 \in \calQ_{C_0}(Q)} \| \Phi^{(q_0)}_1 \phi_2 \phi_3 \|^p_{L^p((1-c)q_0)} \\
& +  \sum_{q_0 \in \calQ_{C_0}(Q)}  
\sum_{q_0' \in \calQ_{C_0}(Q) \setminus \{q_0\}} 
\| \Phi^{(q_0)}_1 \phi_2 \phi_3 \|^p_{L^p((1-c)q_0')} \\
& \leq \sum_{q_0 \in \calQ_{C_0}(Q)} \| \Phi^{(q_0)}_1 \phi_2 \phi_3 \|^p_{L^p((1-c)q_0)} 
+  \left( C c^{-C} R^{ \frac{n+4}2(\frac1p-\frac32 \cdot \frac{n+2}{n+4})+}  \Pi_{i=1}^3 M^\frac{1}2(\phi_i) \right)^p. 
\end{split}
\end{equation*}

We owe an explanation on the how we obtained the last term in the sequence of inequalities above.
Based on the property \eqref{Phia2} of tables we conclude that, for each $q_0, q_0'$ with $q_0 \ne q_0'$ the following hold true
\[
\| \Phi_1^{(q_0)} \phi_2 \phi_3 \|_{L^1((1-c)q_0')}  \leq C c^{-C} R^{-\frac{n-2}4+} \Pi_{i=1}^3 M^\frac12(\phi_i).
\]
Given the size of the cubes, we easily obtain an $L^2$ estimate of type $\| \phi \|_{L^2(q_0')} \leq R^\frac12 M(\phi)$ for each term,
form which it follows that
\[
\| \Phi^{(q_0)} \phi_2 \phi_3 \|_{L^\frac23((1-c)q_0')}  \leq C R^{\frac32} \Pi_{i=1}^3 M^\frac12(\phi_i).
\]
By interpolation, we obtain the $L^p$ bound
\[
\| \Phi^{(q_0)}_1 \phi_2 \phi_3 \|_{L^p((1-c)q_0')}  \leq C c^{-C} R^{\frac{n+4}2(\frac1p-\frac32 \cdot \frac{n+2}{n+4})+}  \Pi_{i=1}^3 M^\frac12(\phi_i), 
\]
which holds true for any $q_0 \ne q_0'$. In conclusion, we have just finished proving the following estimate:
\begin{equation*}
\begin{split}
\| \phi_1 \phi_2 \phi_3 \|^p_{L^p(I^{c,C_0}(Q))}  \leq \sum_{q_0 \in \calQ_{C_0}(Q)} \| \Phi^{(q_0)}_1 \phi_2 \phi_3 \|^p_{L^p((1-c)q_0)} 
+ \left( C c^{-C} R^{ \frac{n+4}2(\frac1p-\frac32 \cdot \frac{n+2}{n+4})+}  \Pi_{i=1}^3 M^\frac{1}2(\phi_i) \right)^p. 
\end{split}
\end{equation*}
Next we repeat the procedure for $\phi_2$ and $\phi_3$: we construct the tables on $\phi_2,\phi_3$, $\Phi_2=\Phi_c(\phi_2,\phi_3,Q), \Phi_3=\Phi_c(\phi_3,\Phi_1,Q)$  with depth $C_0$ and estimate in a similar manner (see Remark \ref{RNLP} after Proposition \ref{NLP}); we obtain
\begin{equation*}
\begin{split}
\| \phi_1 \phi_2 \phi_3 \|^p_{L^p(I^{c,C_0}(Q))} \leq  \sum_{q_0 \in \calQ_{C_0}(Q)} \| \Phi^{(q_0)}_1 \Phi_2^{(q_0)} 
\Phi_3^{(q_0)} \|^p_{L^p((1-c)q_0)} 
+ \left( C c^{-C} R^{\frac{n+4}2(\frac1p-\frac32 \cdot \frac{n+2}{n+4})+}  \Pi_{i=1}^3 M^\frac{1}2(\phi_i) \right)^p. 
\end{split}
\end{equation*}
We recall that $q_0$ has size $\frac{4R}{2^{C_0}} \leq \frac{R}2$ (which in fact may be seen as setting the threshold needed for $C_0$).
Using this we conclude that
\[
\begin{split}
\| \phi_1 \phi_2 \phi_3 \|^p_{L^p(I^{c,C_0}(Q))} &  \leq  \sum_{q_0 \in \calQ_{C_0}(Q)} \left( \bar A_p(\frac{R}2) \right)^p
\Pi_{i=1}^3 M(\Phi_i^{(q_0)})^\frac{p}2  \\
& +  \left( C c^{-C} R^{\frac{n+4}2(\frac1p-\frac32 \cdot \frac{n+2}{n+4})+}  \Pi_{i=1}^3 M^\frac{1}2(\phi_i) \right)^p \\
& \leq \left( \bar A_p(\frac{R}2) \right)^p   \Pi_{i=1}^3 \left( \sum_{q_0 \in \calQ_{C_0}(Q)} M(\Phi^{(q_0)}) \right)^\frac{p}2  \\
& +  \left( C c^{-C} R^{\frac{n+4}2(\frac1p-\frac32 \cdot \frac{n+2}{n+4})+}  \Pi_{i=1}^3 M^\frac{1}2(\phi_i) \right)^p \\
& \leq  \left( \bar A_p(\frac{R}2) \right)^p  \Pi_{i=1}^3 M(\Phi_i)^\frac{p}2 
+   \left( C c^{-C} R^{\frac{n+4}2(\frac1p-\frac32 \cdot \frac{n+2}{n+4})+}  \Pi_{i=1}^3 M^\frac{1}2(\phi_i) \right)^p \\
& \leq (1+cC)^{3p} \left( \bar A_p(\frac{R}2) \right)^p \Pi_{i=1}^3 M(\phi_i)^\frac{p}2 
 +  \left( C c^{-C} R^{\frac{n+4}2(\frac1p-\frac32 \cdot \frac{n+2}{n+4})+}  \Pi_{i=1}^3 M^\frac{1}2(\phi_i) \right)^p.
\end{split}
\] 
where we have used \eqref{PhiM2} in the last line. In using the induction-type bound on $\Pi_{i=1}^3 \Phi_i^{(q_0)}$ we are
using the margin bounds on $\Phi_i$ from \eqref{MPhi2} to conclude with \eqref{mrpg}; this is easily seen to be the case
provided $R$ is large enough to satisfy $CR^{-\frac12} \leq R^{-\frac14}$. In passing to the second estimate we have used the following
sequence inequality
\[
 \sum_i a_i^q b_i^q c_i^q \leq \left( \sum_i a_i \right)^q \left( \sum_i b_i \right)^q  \left( \sum_i c_i \right)^q 
\]
which holds true for $q \geq \frac13$. This easily follows from interpolating the trivial bounds 
\[
\| a b c \|_{l^\frac13},  \| a b c \|_{l^\infty} \leq \| a \|_{l^1} \| b \|_{l^1}  \| c \|_{l^1}. 
\]
where by $abc$ we mean the sequence $\{ a_i b_i c_i \}_{i}$. 

Recalling \eqref{aux21}, we obtain that for any cube $Q_R$ of size $R$ the following holds true
\[
\begin{split}
\| \phi_1 \phi_2 \phi_3 \|_{L^p(Q_R)} & \leq (1+cC)  \Pi_{i=1}^3 M^\frac{1}2(\phi_i) \\
& \cdot \left( (1+cC)^{3p} \left( \bar A_p(\frac{R}2) \right)^p   + 
\left( C c^{-C} R^{\frac{n+4}2(\frac1p-\frac32 \cdot \frac{n+2}{n+4})+}  \right)^p
 \right)^\frac1p 
\end{split}
\]
As a consequence we obtain 
\[
A_p(R) \leq (1+cC) \left( (1+cC)^{3p} \left( \bar A_p(\frac{R}2) \right)^p   + 
\left( C c^{-C} R^{\frac{n+4}2(\frac1p-\frac32 \cdot \frac{n+2}{n+4})+}  \Pi_{i=1}^3 \right)^p \right)^\frac1p.
\]
after redefining $C$, and this is precisely the statement in \eqref{ApR}. 
\end{proof}

\section{The optimality of $p(k)$} \label{pk}

In this section we provide a generalization of the "squashed-cap" construction from \cite{FoKl} with the purpose of establishing the optimality of $p(k)$. To keep things simple, we assume that the surfaces are subsets of the $\S^{n} \subset \R^{n+1}$, the unit sphere in $\R^{n+1}$, and moreover that they are neighborhoods of the points $e_i$; here $e_i=(0,..,0,1,0,..0)$ is the unit vector whose coordinates are $1$ on the $i$'th position and $0$ elsewhere. 

For each $1 \leq i \leq k$, we construct the rectangular parallelepipeds 
\[
D_i=\{  \zeta \in \R^{n+1}, |\zeta_i -1| < \epsilon^2, | \zeta_j| < \epsilon^2, j = \{1,..,k \} \setminus \{ i \}, 
|\zeta_l|  < \epsilon, k+1 \leq l \leq n+1 \}.
\]
We let $S_i = \S^{n} \cap D_i$ and note that $e_i$ is transversal to $S_i$. Thus $S_i$ can be easily parametrized by 
$\Sigma_i: U_i \subset \{ \zeta_i=0 \} \rightarrow S_i$ by 
\[
\Sigma_i(\zeta_1,..,\bar \zeta_{i},..,\zeta_{n+1})= (\zeta_1,..,\zeta_{i-1},\sqrt{1-\zeta_1^2-..-\zeta_{i-1}^2-\zeta^2_{i+1}-..-\zeta^2_{n+1}},\zeta_{i+1},..,\zeta_{n+1}). 
\]
Here, the set $U_i$ is characterized by $| \zeta_j| < \epsilon^2, j = \{1,..,k \} \setminus \{ i \}, \zeta_i=0,
|\zeta_l|  < \epsilon, k+1 \leq l \leq n+1$. We pick $f_i=1$ on $U_i$, therefore
\[
\| f_i \|_{L^2(U_i)} = |U_i|^\frac12 \approx (\epsilon^{2(k-1)} \epsilon^{n+1-k})^\frac12= \epsilon^{\frac{n+k-1}2}
\]
where $|U_i|$ is the measure of $U_i \subset \R^{n}$. 

Next we consider the rectangular parallelepiped $R \subset \R^{n+1}$ described by
\[
R_c=\{  x \in \R^{n+1}, |x_i| \leq c \epsilon^{-2}, i = \{1,..,k \}, |x_i|  < c \epsilon^{-1}, 
k+1 \leq i \leq n+1 \}.
\]
Recall that
\[
\calE_i f_i (x) = \int_{U_i} e^{i x\cdot \Sigma_i(\xi)} d\xi.
\]
By choosing $c$ small enough so that $|x\cdot \Sigma_i(\xi) - x_i | \ll 1$, for all $\xi \in U_i, x \in R_c$, it follows that
\[
|\calE_i f_i (x)| \ges |U_i| \approx \epsilon^{2(k-1)} \cdot \epsilon^{n+1-k} = \epsilon^{n+k-1}, \forall x \in R_c. 
\]
Note that the smallness of $c$ above is universal, and, most important, it is independent of $\epsilon$. 
Then we have
\[
\| \Pi_{i=1}^k \calE_i f_i \|_{L^p} \ges (\epsilon^{n+k-1})^k |R_c|^\frac1p 
\approx \epsilon^{k(n+k-1)} \cdot (\epsilon^{-2k} \epsilon^{-(n+1-k)})^\frac1p = 
\epsilon^{k(n+k-1)} \cdot \epsilon^{-\frac{n+1+k}p}
\]
Then our results would imply
\[
\epsilon^{k(n+k-1)} \cdot \epsilon^{-\frac{n+1+k}p} \les \epsilon^{k \cdot \frac{n+k-1}2}
\]
which, by taking $\epsilon \rightarrow 0$, implies $\frac{k(n+k-1)}2 - \frac{n+k+1}p \geq 0$, thus 
$p \geq \frac{2(n+k+1)}{k(n+k-1)}$. 

The above construction shows the optimality of $p(k)$ in the case of subsets of the sphere and it can be easily adapted
to the case of the paraboloid. In fact, it seems to be relevant for hypersurfaces with non-vanishing Gaussian curvature.

But this counterexample works the same way for multi-conical surfaces based on the following simple observation:
all that is needed is the consistency of scale $\epsilon$ in the $\zeta_{k+1},..,\zeta_{n+1}$ directions with the $\epsilon^2$ scale
in the normal direction at each piece; this is a consequence of the quadratic character of the surfaces in the directions 
$\zeta_{k+1},..,\zeta_{n+1}$.  The details are left as an exercise. 

\bibliographystyle{amsplain} \bibliography{HA-refs}

\end{document}